\documentclass{gtpart}

\usepackage{graphicx}
\usepackage{color}
\usepackage{pinlabel}
\usepackage{palatino}
\usepackage{parskip}
\usepackage{amsthm,amsfonts,amsmath,amssymb,amscd}
\usepackage{verbatim}
\usepackage{enumerate}
\usepackage{tikz}
\usetikzlibrary{arrows}
\usetikzlibrary{trees}
\usetikzlibrary{matrix}
\usetikzlibrary{positioning}
\usepackage{cancel}
\definecolor{darkg}{RGB}{0,128,0}

\newtheorem{thm}{Theorem}[subsection]
\newtheorem{lma}[thm]{Lemma}
\newtheorem{cor}[thm]{Corollary}
\newtheorem{prp}[thm]{Proposition}

\theoremstyle{remark}
\newtheorem{rmk}[thm]{Remark}

\theoremstyle{definition}
\newtheorem{dfn}[thm]{Definition}
\newtheorem{exm}[thm]{Example}
\newtheorem{setup}[thm]{Setup}

\newcommand{\mb}[1]{\mathbf{#1}}
\newcommand{\CC}{\mb{C}}
\newcommand{\FF}{\mb{F}}
\newcommand{\PP}{\mb{P}}
\newcommand{\QQ}{\mb{Q}}
\newcommand{\RR}{\mb{R}}
\newcommand{\ZZ}{\mb{Z}}
\newcommand{\OP}{\operatorname}

\newcommand{\cp}[1]{\CC\PP^{#1}}
\newcommand{\rp}[1]{\RR\PP^{#1}}

\newcommand{\id}{\mathrm{id}}
\newcommand{\kk}{k}
\newcommand{\hf}{\mathrm{hf}}
\newcommand{\Tw}{\mathrm{Tw}}

\newcommand{\Cone}[1]{\OP{Cone}\left(#1\right)}
\newcommand{\TO}[3]{
#1\xrightarrow{#2} #3
}

\newcommand{\QH}{QH}
\newcommand{\HF}{HF}
\newcommand{\CF}{CF}
\newcommand{\HH}{\OP{HH}}
\newcommand{\CW}{CW}

\newcommand{\aA}{\mathcal{A}}
\newcommand{\bB}{\mathcal{B}}

\newcommand{\fF}{\mathcal{F}}
\newcommand{\gG}{\mathcal{G}}

\newcommand{\wW}{\mathcal{W}}

\mathchardef\mhyphen="2D
\newcommand{\nufun}{\mathfrak{nu\mhyphen fun}}

\newcommand{\bimod}[1]{(#1,#1)\mhyphen \mathrm{bimod}}

\mathchardef\ordinarycolon\mathcode`\:
\mathcode`\:=\string"8000
\begingroup \catcode`\:=\active
  \gdef:{\mathrel{\mathop\ordinarycolon}}
\endgroup

\makeatletter
\newcommand*{\da@rightarrow}{\mathchar"0\hexnumber@\symAMSa 4B }
\newcommand*{\da@leftarrow}{\mathchar"0\hexnumber@\symAMSa 4C }
\newcommand*{\xdashrightarrow}[2][]{%
  \mathrel{%
    \mathpalette{\da@xarrow{#1}{#2}{}\da@rightarrow{\,}{}}{}%
  }%
}
\newcommand{\xdashleftarrow}[2][]{%
  \mathrel{%
    \mathpalette{\da@xarrow{#1}{#2}\da@leftarrow{}{}{\,}}{}%
  }%
}
\newcommand*{\da@xarrow}[7]{%
  \sbox0{$\ifx#7\scriptstyle\scriptscriptstyle\else\scriptstyle\fi#5#1#6\m@th$}%
  \sbox2{$\ifx#7\scriptstyle\scriptscriptstyle\else\scriptstyle\fi#5#2#6\m@th$}%
  \sbox4{$#7\dabar@\m@th$}%
  \dimen@=\wd0 %
  \ifdim\wd2 >\dimen@
    \dimen@=\wd2 %
  \fi
  \count@=2 %
  \def\da@bars{\dabar@\dabar@}%
  \@whiledim\count@\wd4<\dimen@\do{%
    \advance\count@\@ne
    \expandafter\def\expandafter\da@bars\expandafter{%
      \da@bars
      \dabar@ 
    }%
  }%
  \mathrel{#3}%
  \mathrel{%
    \mathop{\da@bars}\limits
    \ifx\\#1\\%
    \else
      _{\copy0}%
    \fi
    \ifx\\#2\\%
    \else
      ^{\copy2}%
    \fi
  }%
  \mathrel{#4}%
}
\makeatother

\newcommand{\TODASH}[3]{
#1\xdashrightarrow{#2} #3
}

\title{Generating the Fukaya categories of \\ Hamiltonian $G$-manifolds}

\author{Jonathan David Evans \newline \ \  Yank\i\ Lekili}
\address{University College London \newline King's College London}

\begin{document}

\begin{abstract}
    Let $G$ be a compact Lie group and $\kk$ be a field of
    characteristic $p \geq 0$ such that $H^*(G)$ has no $p$-torsion if
    $p>0$. We show that a free Lagrangian orbit of a Hamiltonian
    $G$-action on a compact, monotone, symplectic manifold $X$
    split-generates an idempotent summand of the monotone Fukaya
    category $\mathcal{F}(X; \kk)$ if and only if it represents a
    non-zero object of that summand (slightly more general results are
    also provided).  Our result is based on: an explicit understanding
    of the wrapped Fukaya category $\mathcal{W}(T^*G; \kk)$ through
    Koszul twisted complexes involving the zero-section and a
    cotangent fibre; and a functor $D^b \mathcal{W}(T^*G; \kk) \to
    D^b\mathcal{F}(X^{-} \times X; \kk)$ canonically associated to the
    Hamiltonian $G$-action on $X$. We explore several examples which
    can be studied in a uniform manner including toric Fano varieties
    and certain Grassmannians.

\end{abstract}

\maketitle

\section{Introduction}

\subsection{Background}

In this paper we develop a technique for studying monotone Fukaya
categories of symplectic manifolds admitting a Hamiltonian action of a
compact Lie group. The Fukaya category of a symplectic manifold is a
triangulated $A_\infty$-category which organises information about
Lagrangian submanifolds and their intersections. It is well-known for
its appearance in Kontsevich's homological mirror symmetry conjecture,
however a more basic motivation for studying the triangulated
$A_\infty$-structure is that it makes information about Lagrangian
intersections accessible to computation. For example, one can make
sense of the notion of a {\em split-generator}, a Lagrangian
submanifold from which (representatives of) all objects of the
category can be obtained by forming iterated cones and taking direct
summands. This has geometric ramifications: if $K$ split-generates
then $K$ is not displaceable from any Lagrangian submanifold $L$ with
$\HF(L,L)\neq 0$. Finding a split-generator is also the first step in
many modern proofs of homological mirror symmetry (for example \cite{AbouzaidSmith,SeidelK3,Sheridan}).

The aim of this paper is to determine when a free Lagrangian orbit of
a Hamiltonian $G$-action split-generates a summand\footnote{In Section
  \ref{sct:qh}, we explain how the quantum cohomology splits as a
  direct sum of local rings and the Fukaya category correspondingly
  splits into summands. This is a finer splitting than the usual
  splitting by eigenvalues of quantum product with the first Chern
  class.}  of the Fukaya category. We show in Theorem
\ref{thm:generation} that, when the characteristic of the coefficient
field is not a torsion prime for $G$, the orbit generates a summand
whenever it represents a non-zero object in that summand. In the final
section of the paper, we give a number of examples where this yields
explicit split-generators, including toric Fano manifolds and certain
Grassmannians.

\subsection{The main result}

Fix a compact Lie group $G$ and a monotone symplectic manifold
$X$. Suppose that there is a Hamiltonian $G$-action on $X$ with
equivariant moment map $\mu\colon X\to\mathfrak{g}^*$. There is a
natural monotone Lagrangian correspondence $C\subset (T^*G)^{-}\times
X^{-}\times X$ called the {\em moment Lagrangian} defined as follows:
\[C:=\{(g,v,x,y)\in (T^*G)^{-}\times X^{-}\times X\ :\ v=\mu(gx),\ y=gx\}.\]
This was noticed by Weinstein \cite{Weinstein} and Guillemin-Sternberg
\cite{GuilleminSternberg}. The authors first learned of it from
Teleman's paper \cite{Teleman}.

Fix a field $\kk$. We will use the correspondence $C$ to define a
triangulated $A_\infty$-functor
\[\mathfrak{C}: D^b \wW(T^*G;\kk) \to D^b \fF(X^{-}\times X;\kk) \]
from the derived wrapped Fukaya category of $T^*G$ to the derived
Fukaya category of monotone Lagrangians in $X^{-} \times X$ (Section
\ref{sct:mcf}).

Abouzaid $\cite{AbouzaidCotFib}$ proves that $D^b \wW(T^*G;\kk)$ is
generated by the cotangent fibre $T_1^*G$ as a triangulated category,
so to define a triangulated $A_\infty$-functor on $D^b\wW(T^*G;\kk)$,
it suffices to define it on the subcategory with a single object
$T^*_1G$. Our functor is defined using the quilt formalism of
\cite{MWW} and the geometric composition theorem
\cite{LekiliLipyanskiy,WWfunctoriality}; we review this formalism in
Section \ref{sct:quilt} and explain how to modify it in our
setting. The key point in defining this functor is the fact that the
cotangent fibre $T^*_1G$ at the identity in $G$ can be composed
geometrically with the correspondence $C$, and the composition is the
diagonal $\Delta\subset X^-\times X$ (Lemma \ref{lma:diagcomp}). In
particular, the functor $\mathfrak{C}$ gives a map
\begin{equation}\label{e2?}
\mathfrak{C}^1\colon\CW^*(T_1^*G, T_1^*G ; \kk) \to \CF^*(\Delta, \Delta; \kk).
\end{equation}
On cohomology, this sends the wrapped Floer cohomology of the
cotangent fibre of $G$ (infinite-dimensional) to quantum cohomology of
$X$ (finite-dim\-en\-si\-on\-al). In general, the functor
$\mathfrak{C}$ is far from being full or faithful.

\begin{rmk}
It is conjectured in \cite{Teleman} that the morphism from (\ref{e2?})
can be upgraded to an $E_2$-algebra morphism.
\end{rmk}

Next, let us assume that $\mu$ is transverse to $0\in\mathfrak{g}^*$
and that $G$ acts freely on $\mu^{-1}(0)$. In this situation, we can
prove that the image $\mathfrak{C}(G)$ of the zero-section $G \subset
T^* G$ is represented in $D^b \mathcal{F}(X^{-} \times X;\kk)$ by the
Lagrangian submanifold: \begin{equation} \label{Mdef} \mathfrak{C}(G)
  \cong \left\{(x,y)\in X^{-}\times
  X\ :\ x,y\in\mu^{-1}(0),\ y=gx\mbox{ for some }g\in
  G\right\}. \end{equation} In particular, when $\mu^{-1}(0)$ is a
free Lagrangian orbit $L$, we have $\mathfrak{C}(G) \cong L \times L$.

To obtain more information about $\mathfrak{C}(G)$, we first study the
domain of our functor $\mathfrak{C}$, namely the wrapped Fukaya
category $\mathcal{W}(T^*G; \kk)$. Wrapped Fukaya categories of
cotangent bundles have been studied extensively in the past few
years. For example, the papers of Fukaya-Seidel-Smith \cite{FSS1,FSS2}
and Abouzaid \cite{AbouzaidCotFib, AbouzaidBasedLoops} provide
concrete ways of understanding the wrapped Fukaya categories of
cotangent bundles: there is a full and faithful embedding of
$\wW(T^*G;\kk)$ into the category of $A_\infty$-bimodules over the
dg-algebra $C_{-*}(\Omega G;\kk)$ of cubical chains on the based loop
space, equipped with the Pontryagin product. Under this embedding, the
zero-section $G \subset T^*G$ is identified with the trivial
$A_\infty$-bimodule $\kk$.

In Section \ref{sct:formality} we prove that if $p=\OP{char}(\kk)$ and
$H^*(G;\ZZ)$ has no $p$-torsion (if $p>0$), then there is a
quasi-isomorphism of $A_\infty$-algebras (Theorem \ref{thm:formal}) :
\[ H_{-*}(\Omega G;\kk) \to C_{-*}(\Omega G;\kk) \]
While this follows from classical arguments, we could not find a proof
in the literature, so we have explained this in some detail
here. Since $H_{-*}(\Omega G;\kk)$ is a polynomial algebra, this
allows us to argue that the standard Koszul complex resolving $\kk$ as
a $H_{-*}(\Omega G;\kk)$-bimodule defines a {\em Koszul twisted
  complex} (Definition \ref{dfn:koszul}) representing $\kk$ in the
category of $A_\infty$-bimodules over $C_{-*}(\Omega G;\kk) \cong
H_{-*}(\Omega G;\kk)$. In $\wW(T^*G;\kk)$, this tells us that $G$ is a
Koszul twisted complex built out of $T^*_1G$.

In Section \ref{sct:twcpx}, we show that Koszul twisted complexes are
pushed forward to Koszul twisted complexes under
$A_\infty$-functors. In particular, applying our functor
$\mathfrak{C}$, we see that $\mathfrak{C}(G)$ can be expressed as a
Koszul twisted complex built out of $\Delta$ in $\fF(X^{-}\times
X;\kk)$. A Koszul twisted complex $K$ built out of an object $K'$ has
the crucial property that, if certain morphisms involved in the
differential are nilpotent, then $K$ \emph{split-generates} $K'$
(Corollary \ref{cor:nilpgen}). This can be seen as a form of Koszul
duality; analogous nilpotence conditions play a decisive role in
establishing the convergence of Eilenberg-Moore spectral sequence
\cite{dwyer}; see also \cite{dwyergreenlees}. For a recent study of
Koszul duality in the context of Fukaya categories see \cite{EL},
which has partially inspired our work.

The complex representing $\mathfrak{C}(G)$ is built out of $\Delta$,
so these morphisms are elements of $\CF(\Delta,\Delta;\kk)$ which, by
the homological perturbation lemma, can be identified with the quantum
cohomology $\QH(X;\kk)$ as an $A_\infty$-algebra with vanishing
differential but possibly non-trivial higher products. In particular,
our main result exploits the algebra structure of $\QH(X;\kk)$ to
determine whether $\mathfrak{C}(G)$ split-generates $\Delta$.

To state our main theorem, recall that we have a block decomposition
of the quantum cohomology $\QH(X;\kk) = \bigoplus_\alpha
\QH(X;\kk)_\alpha $ into local rings (Section \ref{sct:qh}), a
corresponding splitting of $\Delta$ into idempotent summands,
$\bigoplus_{\alpha\in W}\Delta_\alpha$ and a decomposition of the
Fukaya category into summands $D^{\pi} \mathcal{F}(X;\kk) =
\bigoplus_\alpha D^{\pi} \mathcal{F}(X;\kk)_\alpha$.  This is a
refinement of the usual splitting of quantum cohomology into
eigen-spaces of quantum product with $c_1(X)$.

In Section \ref{sct:ham}, we prove our main theorem:

\begin{thm}\label{thm:mainthm}
  Suppose $G$ is a compact Lie group and $\kk$ is a field of characteristic $p \geq 0$ such that
  $H^*(G)$ has no $p$-torsion (if $p>0$). Let $X$ be a compact,
  monotone, Hamiltonian $G$-manifold such that the associated moment
  map $\mu: X \to \mathfrak{g}^*$ has $0$ as a regular value and that
  $G$ acts freely on $\mu^{-1}(0)$.
    
  Then, the Lagrangian submanifold $\mathfrak{C}(G) \subset X^{-}
  \times X$, defined by (\ref{Mdef}), split-generates $\Delta_\alpha$
  in $\mathcal{F}(X^{-} \times X)$ if and only if
  $HF(\mathfrak{C}(G),\Delta_\alpha; \kk) \neq 0$.

  In particular, if $\mu^{-1}(0)= L$ is a free Lagrangian orbit which
  represents a non-zero object in the summand $D^{\pi} \mathcal{F}(X;
  \kk)_\alpha$, then $L$ split-generates the Fukaya category
  $D^{\pi} \mathcal{F}(X;\kk)_\alpha$.
\end{thm} 

The last part of our theorem states if $\mu^{-1}(0)= L$ is a free
Lagrangian orbit, then either $L$ projects to a zero object in the
summand $D^{\pi} \mathcal{F}(X;\kk)_\alpha$ or it split-generates the
summand. We shall see in Section \ref{sct:examples} that both
situations may arise.  Determining the summands to which $L$ projects
non-trivially boils down to computing $HF(L,L)$ as a module over the
ring of semi-simple elements in quantum cohomology (see Remark
\ref{boilsdown}).

We expect that there will be straightforward generalisations of our
theorem to the case where $G$ acts on $\mu^{-1}(0)$ with finite
stabilisers; in this case the immersed Lagrangian $\mathfrak{C}(G)$
should be equipped with a bounding cochain in the sense of
{\cite[Ch. 3]{FOOO}}.  It would also be interesting to study the
functor $\mathfrak{C}$ when $X$ is non-compact. However, in this case
the correspondence $C$ would be non-compact and one has to arrange
that holomorphic curve theory is well-behaved.

\begin{rmk}\label{rmk:othergeneration}
Another curious consequence of our theorems is the fact that if $L$ is
a free Lagrangian orbit which split-generates as in Theorem
\ref{thm:mainthm} then any other non-zero object in $D^{\pi}
\mathcal{F}(X;\kk)_\alpha$ split-generates $D^{\pi}
\mathcal{F}(X;\kk)_\alpha$. See Corollary \ref{cor:othergeneration}.
\end{rmk}

\begin{rmk} \label{rmk:afooogen}
  Abouzaid, Fukaya, Oh, Ohta and Ono \cite{AFOOO} prove that a
  Lagrangian $L$ split-generates a summand $D^{\pi} \fF(X;\kk)_\alpha$
  of the Fukaya category if the closed-open map
  \[ \mathcal{CO} \colon\QH^\bullet (X;\kk)_\alpha\to \HH^\bullet (\CF(L_\alpha,L_\alpha; \kk)) \]
  is injective. For a proof of this in the monotone case see
  {\cite[Corollary 2.18]{Sheridan}}.  Our proof does not appeal to
  this result, however, we observe that our generation result implies
  that $\QH^\bullet (X;\kk)_\alpha \cong \HH^\bullet
  (\CF(L_\alpha,L_\alpha;\kk))$ (see Corollary \ref{cor:aboucrit}).
\end{rmk}

\subsection{Examples and applications}

We now present a selection of examples where Theorem
\ref{thm:generation} can be applied.  We defer the study of more
examples to future work in the context of mirror symmetry.

The following is an immediate corollary our main theorem coupled with
of computations of Floer cohomology due to Cho-Oh \cite{ChoOh} and
Fukaya-Oh-Ohta-Ono \cite{FOOOtoric}:

\begin{cor}[Corollary \ref{cor:toricgen}]
  Let $X$ be a toric Fano variety and $\kk$ be an algebraically closed
  field of arbitrary characteristic.  Any summand $D^{\pi}
  \mathcal{F}(X;\kk)_\alpha$ of the Fukaya category is generated by
  the barycentric (monotone) torus fibre $L$ equipped with an
  appropriate flat $\kk$-line bundle $\xi_\alpha$.
\end{cor}

Note that, since $\mathcal{CO}$ is a unital ring homomorphism, in view
of Remark \ref{rmk:afooogen}, it is immediate {\cite[Corollary
    2.18]{Sheridan}} that if $\QH(X;\kk)_\alpha$ is isomorphic to a field (i.e. a finite extension of $\kk$) and $L$ represents a non-zero
object in $D^{\pi} \mathcal{F}(X;\kk)_\alpha$, then $L$
split-generates $D^{\pi} \fF(X;\kk)_\alpha$. The more interesting and
difficult case for proving split-generation is therefore when the
quantum cohomology is not semi-simple, that is, it does not split as a
direct sum of fields. The above corollary proves split-generation in
all cases. For example, it applies in the case of the Ostrover-Tyomkin
toric Fano four-fold whose quantum cohomology was demonstrated to be
non-semi-simple over $\CC$ \cite{OstroverTyomkin}.

Since we work with monotone symplectic manifolds, the quantum
cohomology can be defined over $\ZZ$ and we can choose to specialise
to any coefficient field. The resulting ring $\QH(X;\kk)$ may be
semi-simple for some values of $\OP{char}(\kk)$ and not for
others. This is the analogue of varieties over $\OP{Spec}(\ZZ)$ having
ramification at $(p)$:

\begin{exm}
  The quantum cohomology of $\cp{n}$ is the
  $\mathbf{Z}/{(2n+2)}$-graded ring
  \[\QH^\bullet (\cp{n};\ZZ)=\ZZ[H]/(H^{n+1}-1), \ \ |H|=2. \]
  If $\kk$ is an algebraically closed field of characteristic $p$, and
  $n+1=p^sq$ ($\gcd(p,q)=1$) then
  \[\QH(\cp{n};\kk)\cong\bigoplus_{\mu\ :\ \mu^q=1}\QH(\cp{n};\kk)_{e_\mu}\]
  where $e_{\mu}$ is the idempotent
  \[\frac{1}{q}\sum_{i=0}^{q-1}\mu^iH^{p^si}.\]
  We see that this is semisimple if and only if $p$ does not divide
  $n+1$ (so that there are $n+1$ roots of unity, and hence $n+1$ field
  summands).
\end{exm}

At first sight, this seems like an artificial way of making the
quantum cohomology non-semi-simple. However, we emphasize that Fukaya
categories over a field of characteristic $p$ may contain geometric
information about Lagrangian submanifolds which is not available if we
work over other fields:

\begin{exm}
Recall that if $L$ is a Lagrangian submanifold with nonzero self-Floer
cohomology and $\mathfrak{m}_0(L)$ is the count of Maslov 2 discs
passing through a generic point of $L$ then $2\mathfrak{m}_0(L)$ is an
eigenvalue of quantum product with $2c_1(X)$. There are many
Lagrangian submanifolds in $\cp{n}$ whose minimal Maslov number is
strictly bigger than 2. For example, $\rp{n}\subset\cp{n}$ ($n\geq
2$), $PSU(n)\subset\cp{n^2-1}$, and many more (see \cite{BedulliGori}
for some further examples). For these Lagrangians, we have
$\mathfrak{m}_0(L)=0$. The first Chern class of $\cp{n}$ is $(n+1)H$,
so its eigenvalues over $\kk$ are $\{(n+1)\mu\ :\mu^{n+1}=1\}$. Zero
is an eigenvalue precisely when $n+1\equiv 0\in\kk$, that is when
$p=\OP{char}(\kk)$ divides $n+1$. We see that these Lagrangian
submanifolds can only define non-zero objects of the Fukaya category
in characteristic $p$.
\end{exm}

\begin{rmk}
In \cite{chiang}, the authors studied an example of a Lagrangian
submanifold in $\cp{3}$ whose Floer cohomology was non-vanishing only
in characteristic 5. However, there the relevant Fukaya category was
semi-simple.
\end{rmk}

We will use our technology to prove the following split-generation
results which crucially rely on working in non-zero characteristic:

\begin{prp}
\begin{enumerate}
\item[(a)] If the real part of a toric Fano manifold is orientable,
  then it split-generates the Fukaya category in characteristic 2
  (Example \ref{exm:realtoric}).
\item[(b)] There is a Hamiltonian $U(n)$-action on the Grassmannian
  $\OP{Gr}(n,2n)$ with a free Lagrangian orbit $L$. If $n$ is a power
  of 2 then $L$ split-generates the Fukaya category in characteristic
  2 (Corollary \ref{cor:grn}).
\item[(c)] There is a Lagrangian $PSU(n)\subset\cp{n^2-1}$. If $n$ is
  a power of $p$ then $PSU(n)$ split-generates the Fukaya category of
  $\cp{n^2-1}$ in characteristic $p$ (Example \ref{exm:psun}).
\end{enumerate}
\end{prp}

These results follow directly from our split-generation criterion and
computations of Floer cohomology (due respectively to Haug
\cite{Haug}, Oh \cite{OhIII} and the authors (example
\ref{exm:psun})).

Another interesting consequence of our results is a non-formality
result for the quantum cohomology ring when it fails to be
semi-simple. Recall that we can equip $\QH(X;\kk)$ with an
$A_\infty$-structure by applying the homological perturbation lemma to
$\CF(\Delta,\Delta;\kk)$.  This $A_\infty$-algebra, which we denote by
$\mathcal{QH}(X;\kk)$, seems only to have been computed in cases where
it is intrinsically formal. There were early expectations
\cite{RuanTian} that it should be formal when $X$ is K\"{a}hler, at
least over $\CC$. However, we show:

\begin{prp}[Corollary \ref{cor:nonform}]
If $X$ is a toric Fano and $\kk$ is an algebraically closed field such
that $\QH(X;\kk)$ is not semisimple as a $\ZZ/2$-graded algebra, then
$\mathcal{QH}(X;\kk)$ is not quasi-isomorphic to $\QH(X;\kk)$.
\end{prp}

For example, the quantum cohomology $A_\infty$-algebra of the
Ostrover-Tyomkin toric Fano four-fold \cite{OstroverTyomkin} is not
formal over $\CC$.

\subsection{Outline of paper}

In Section \ref{sct:formality}, we collect together results on the
formality of compact Lie groups and their based loop spaces over a
field of arbitrary characteristic; in particular, we show that Lie
groups are formal in characteristic $p$ if their cohomology ring has
no $p$-torsion.

In Section \ref{sct:twcpx}, we develop the notion of a Koszul twisted
complex and show that, in the derived wrapped Fukaya category of the
cotangent bundle of a compact Lie group with no $p$-torsion, the
zero-section can be written as a Koszul twisted complex built out of
the cotangent fibre.

In Section \ref{sct:qh}, we review the decomposition of the quantum
cohomology into idempotent summands and the corresponding
decomposition of the Fukaya category.

In Section \ref{sct:quilt}, we review the standard quilt theory for
constructing $A_\infty$-functors from Lagrangian correspondences and
adapt it to the situation where some of the Lagrangians are
non-compact (but the correspondences are still compact).

In Section \ref{sct:ham}, we state and prove all the main results on
generation.

Finally, in Section \ref{sct:examples}, we find split-generators in
some explicit examples and give further applications.

{\bf Acknowledgments: } We thank Denis Auroux, Paul Biran, Janko
Latschev, Ivan Smith, Dmitry Tonkonog and Chris Woodward for their
interest in this work. We also thank an anonymous referee for their
helpful comments and corrections. Y.L. is partially supported by the
Royal Society and an NSF grant DMS-1509141.

\section{Formality results for compact Lie groups}\label{sct:formality}

In this section, we collect together formality results for $C^*(G)$
and $C_{-*}(\Omega G)$ as $A_\infty$-algebras, over a field of
arbitrary characteristic, where $G$ is a compact, connected Lie group.

\subsection{Lie groups and products of spheres}

Let $G$ be a compact, connected Lie group (over $\RR$). It is
well-known that any such $G$ is a finite quotient of a product of a torus 
and a simple non-abelian Lie group by a central subgroup. Recall that a
compact connected Lie group $G$ is called \emph{simple} if it is
non-abelian and has no closed proper normal subgroup of positive
dimension. For convenience, let us also recall the classification
theorem of connected, compact, simple Lie groups as given by the
following table:

\begin{figure}[htbp!] 
\centering 
\begin{tikzpicture}1
\matrix (mymatrix) [matrix of nodes, nodes in empty cells, text height=1.5ex, 
	column 1/.style={minimum width=3.5},
	    column 2/.style={minimum width=4.5},
	        column 3/.style={minimum width=6.6},
		    column 4/.style={minimum width=4.5} , align=right]
{
                &    $\dim G$    & Linear group & Universal cover & Centre &  $\frac{\dim G}{\OP{rank} G}-1$ \\
$A_n (n\geq 1)$ & $n(n+2)$  & $SU(n+1)$      &                 & $\ZZ_{n+1}$& $n+1$&\\
$B_n (n\geq 2)$ & $n(2n+1)$ & $SO(2n+1)$     & $Spin(2n+1)$      & $\ZZ_2$ & $2n$ &\\
$C_n (n\geq 3)$ & $n(2n+1)$ & $Sp(n)$        &                 & $\ZZ_2$ & $2n$ &  \\
$D_n (n\geq 4)$ & $n(2n-1)$ & $SO(2n)$       & $Spin(2n)$        & $\ZZ_2 \ltimes \ZZ_2$ & $2n-2$ &\\
$ G_2 $         & $14$      & & &1 &6 & \\
$ F_4 $         & $52$      & & &1  &12 & \\
$ E_6 $         & $78$      & & &$\ZZ_3$ &12 & \\
$ E_7 $         & $133$     & & &$\ZZ_2$ &18& \\
$ E_8 $         & $248$     & & &1 & 30&\\ };      
\draw (mymatrix-1-1.south west) ++ (-0.7cm,0) -- (mymatrix-1-6.south east);
\end{tikzpicture}
\caption{Classification of connected, compact, simple Lie groups (The
  centre of $SO(2n)$ is $\ZZ_4$ if $n$ is odd, $\ZZ_2 \oplus \ZZ_2$ if
  $n$ is even) }

\label{classification}
\end{figure}

There is an extensive literature on the homotopy theory of Lie
groups. We recommend the excellent survey article \cite{mimura} by
Mimura. In particular, let us recall the classical theorem of Hopf
\cite{hopf} which states that, over $\QQ$, the singular cohomology of
a compact, simple group $G$ is given by: \[ H^*(G;\QQ) = \Lambda (
x_{2n_1-1},x_{2n_2-1}, \ldots x_{2n_l-1}) \ \text{\ with\ }
|x_{2n_i-1}| = 2n_i - 1 \] where $\text{rank\ } G=l$, $\text{dim\ } G
= \sum_{i=1}^l (2n_i-1)$, and $(n_1,\ldots, n_l)$ is called the type
of $G$. In fact, $G$ has the rational homotopy type of a product of
odd-dimensional spheres:
\[ G \underset{\QQ}{\cong} \prod_{i=1}^l S^{2n_i-1}  \]
Indeed, in \cite{serre}, Serre constructs a map $f : G \to \prod
S^{2n_i-1}$ such that the induced map of dg-algebras:\[ f^{*} :
C^{*}\left( \prod S^{2n_i-1}; \QQ\right) \to C^*(G; \QQ)\] is a
quasi-isomorphism.

In characteristic $p>0$, the story is more complicated. First of all,
in general, cohomology groups of compact, connected, simple Lie groups
can have torsion (see \cite{mimura}). Nonetheless, the classical Lie
groups $U(n),SU(n), Sp(n)$ have no torsion:
\begin{align*}
	H^*(U(n) ; \ZZ)  &= \Lambda(x_1,x_3,\ldots,x_{2n-1}) \\
	H^*(SU(n) ; \ZZ ) &= \Lambda(x_3,x_5,\ldots,x_{2n-1}) \\
	H^*(Sp(n); \ZZ) &= \Lambda(x_3,x_7,\ldots,x_{4n-1}) 
\end{align*} 

However, even for these groups, the mod $p$ homotopy type is not the
same as the mod $p$ homotopy type of a product of odd-dimensional
spheres for all $p$ as can be detected by non-triviality of Steenrod
operations. For example, $SU(3)$ is not homotopy equivalent to $S^3
\times S^5$ mod $2$.

In general, for $p$ sufficiently large, Serre \cite{serre} (for
classical groups) and Kumpel \cite{kumpel} (for exceptional groups),
proved that the mod $p$ homotopy type of a compact connected simple
Lie group $G$ is the same as a product of odd-spheres:

\begin{thm} (Serre \cite{serre}, Kumpel \cite{kumpel})
  Let $G$ be a compact connected simple Lie group of type
  $(n_1,\ldots,n_l)$, and $p$ be a prime then
  \[ G \underset{p} \cong \prod S^{2n_i-1} \]
  if and only if $p \geq ( \text{dim\ } G / \text{rank\ } G) -1$.
\end{thm}

In particular, for $p \geq ( \text{dim\ } G / \text{rank\ } G) -1$
there exists a map $f: G \to \prod S^{2n_i-1}$ such that the induced
map of dg-algebras:
\[ f^{*} : C^{*}\left( \prod S^{2n_i-1}; \FF_p\right) \to C^*(G; \FF_p). \]
is a quasi-isomorphism.

\subsection{Formality}

Throughout, we will treat a dg-algebra as a special case of an
$A_\infty$-algebra which has vanishing higher products. In this paper,
the following notion of formality for $A_\infty$-algebras will play a
key role:

\begin{dfn}
  Let $C^*$ be a graded $A_\infty$-algebra defined over $\ZZ$. Let us
  write $C^*_{\QQ}$ and $C^*_{p}$ for $p$ prime, the graded
  $A_\infty$-algebras obtained from $C^*$ by tensoring it (in the
  derived sense) with $\QQ$ and $\FF_p$ respectively. We say that
  $C^*$ is formal mod $p$ if there exists an
  $A_\infty$-quasi-isomorphism:
  \[ H(C^*_p) \to C^*_p \] 
  where we view $H(C^*)$ as an $A_\infty$-algebra with vanishing
  differential and higher products. Similarly, we say that $C^*$ is
  formal over $\QQ$ if the same condition holds over $\QQ$.

  For a topological space $X$, we say that $X$ is formal mod $p$
  (respectively over $\QQ$) if the singular cochain complex $C^*(X)$
  (viewed as an $A_\infty$-algebra) is formal mod $p$ (respectively
  over $\QQ)$.
\end{dfn}

\begin{rmk}
  We note that over $\QQ$ the singular cochain complex $C^*(X)$ for a
  topological space can in fact be modelled by a $C_\infty$-algebra
  $X$ and for simply connected spaces of finite-type the $C_\infty$
  quasi-isomorphism type of this is a complete invariant of the
  rational homotopy type of $X$ (see \cite{kadeishvili}). On the other
  hand, we do not know if two $C_\infty$-algebras over $\QQ$ that are
  quasi-isomorphic as $A_\infty$-algebras are also quasi-isomorphic as
  $C_\infty$-algebras (most probably, this is not true).
\end{rmk}

\begin{exm} \label{spheres}
  A sphere $S^n$ is formal mod $p$ for all $p$, and over
  $\QQ$. Picking cochain representative for the fundamental class
  gives us a linear map that is a quasi-isomorphism:
  \[ \Lambda (x) = H^*(S^n) \to C^*(S^n), \text{ with } |x|=n \]
  Furthermore, using a normalized simplicial chain model for
  $C^*(S^n)$ one sees that $x^2=0$ at the chain
  level. Hence, this map can be taken to be a dg-algebra map.
\end{exm}
\begin{exm} \label{product}
  If $X$ and $Y$ are formal mod $p$ (respectively over $\QQ$), then $X
  \times Y$ is formal mod $p$ (respectively over $\QQ$). This follows
  from the Eilenberg-Zilber theorem which gives a dg-algebra
  quasi-isomorphism over $\ZZ$ {\cite[Section 17]{eilenbergmoore}}:
  \[ C^*(X \times Y) \to C^*(X) \otimes C^*(Y) \]
  and the K\"unneth theorem, which identifies $H^*(X \times Y)$ with
  $H^*(X) \otimes H^*(Y)$ over a field. We remark that the
  Eilenberg-Zilber theorem establishes quasi-isomorphisms between
  $C^*(X) \otimes C^*(Y)$ and $C^*(X \times Y)$ in each direction but
  only one of these can be taken to be a dg-algebra
  quasi-isomorphism. Indeed, there is an $A_\infty$-algebra
  quasi-isomorphism from $C^*(X) \otimes C^*(Y) \to C^*(X \times Y)$
  but not a dg-algebra quasi-isomorphism (cf. Munkholm {\cite[Theorem
      2.6]{munkholm}}).

  As a particular example, let us note that $G=T^n$ is formal mod $p$
  for all $p$, and over $\QQ$.
\end{exm}

As a consequence of the previous discussion, we immediately conclude
the following:

\begin{cor} \label{aboveresult}
  A connected, compact, simple Lie group $G$ is formal over $\QQ$, and
  it is formal mod $p$ for $p \geq \text{dim } G / \text{rank G} - 1$.
\end{cor} 

\begin{rmk}
 Another way to see that Lie groups are formal in characteristic 0 is
 to consider the dg subalgebra of the de Rham complex consisting of
 invariant forms. These are precisely the harmonic forms, with respect
 to an invariant metric, therefore the subspace of harmonic forms is
 closed under multiplication. By the Hodge theorem, the inclusion of
 the subspace of harmonic forms is a quasi-isomorphic embedding of the
 cohomology into the de Rham cochains, proving formality. (The de Rham
 complex is {\em geometrically formal}.)
\end{rmk}

Note that Corollary \ref{aboveresult} is obtained via the
characterization of mod $p$ homotopy type of the group $G$. By a
theorem of Mandell \cite{mandell}, the mod $p$ homotopy type of a
simply-connected finite-type space $X$ can be recovered from the
$E_\infty$-algebra structure on $C^*(X;\FF_p)$. However, our interest
is in a much coarser invariant of $X$, namely the $A_\infty$-algebra
(i.e $E_1$-algebra) structure of $C^*(X;\FF_p)$ up to
quasi-isomorphism of $A_\infty$-algebras.  In particular, note that
the dg-algebra structure of $C^*(X;\FF_p)$ is not enough to construct
Steenrod operations.  Therefore, it is possible for a space $X$ that
$C^*(X;\FF_p)$ is formal as an $A_\infty$-algebra mod $p$ without
being formal as an $E_\infty$-algebra mod $p$. We next address this
delicate formality question for the connected, compact, simple Lie
groups $G$ and primes $p$ such that $H^*(G)$ does not have
$p$-torsion. In view of the above corollary, our result is of
particular interest when $p < \text{dim } G/ \text{rank G}-1$.

We recall that $\Omega G$ is the based loop space of $G$. We denote by
$C_*(\Omega G)$ the normalized singular chain complex. To be consistent, we
prefer to use instead the cohomologically graded complex
$C_{-*}(\Omega G)$ which is concentrated in non-positive degrees. This
has a dg-algebra structure coming from the Pontryagin product on
$\Omega G$ given by concatenation of loops.

\begin{thm} \label{thm:formal}
Let $G$ be a connected, compact, simple Lie group and $p$ be a prime
such that $H^*(G)$ has no $p$-torsion:
\begin{enumerate}[(i)]
\item The $A_\infty$-algebra $C_{-*}(\Omega G)$ is formal mod $p$.
\item The $A_\infty$-algebra $C^*(G)$ is formal mod $p$. 
\end{enumerate}
\end{thm} 

We first give a proof of (i) which holds in slightly more general
setting. Recall that if $H^*(G)$ has no $p$-torsion, then over a field
$\kk$ of characteristic $p$, $H_{-*}(\Omega G; \kk)$ is isomorphic to
a polynomial algebra.

\begin{prp} \label{shm}
  Let $\kk$ be a field. Let $G$ be an H-space such that
  \[ H_{-*}(\Omega G) \cong \kk[y_1,y_2,\ldots y_n] \] 
  Then, the differential graded algebra $C_{-*}(\Omega G)$ is formal
  over $\kk$.
\end{prp}

\begin{proof}
  The crucial ingredient is that $C_{-*}(\Omega G)$ is more than just
  an $A_\infty$-algebra: it comes equipped with an $A_\infty$-morphism
  \[\Phi\colon C_{-*}(\Omega G) \otimes_{\kk} C_{-*} (\Omega G) \to C_{-*}(\Omega G) \] 
  which makes it into a strongly homotopy commutative algebra (see
	Clark \cite[Corollaries 1.7 and 3.5]{clark} and \cite[Theorem 4.3]{sugawara}). In particular, one
  has $\Phi(x \otimes 1) = \Phi(1 \otimes x) = x$ for all $x$ where
  $1$ denotes the 0-chain at the identity element of the group $\Omega
  G$.  Conceptually, the existence of this $A_\infty$-morphism is due
  to the fact that $\Omega G$ is homotopy equivalent to the
  double-loop space $\Omega^2 B G$ (cf. \cite[Chapter
    2]{adamsinfinite}). One can use this iteratively to define an
  $A_\infty$-morphism
	\[ \Phi^{[n]} : \colon C_{-*} (\Omega G)^{\otimes n} \to C_{-*}(\Omega G) \] 
	for any $n \geq 2$ (see \cite[Section 4.4]{munkholm}) defined by
	\[ \Phi^{[n]} = \Phi \circ (\Phi^{[n-1]} \otimes \id). \]
	It follows that for all $x$, we have \[ \Phi^{[n]}(1 \otimes \ldots
  \otimes x \otimes \ldots \otimes 1) = x. \]

  We next borrow an argument from {\cite[Section 7.2]{munkholm}}. Let
  $c_i \in C_{-*}(\Omega G)$ be chains representing the cohomology
  classes $y_i=[c_i]$. Consider the dg-algebra map
  \[ H_{-*} (\Omega G) \to C_{-*}(\Omega G)^{\otimes n} \] 
  given by \[ y_i \to 1 \otimes \ldots \otimes c_i \otimes \ldots
  \otimes 1 \] where $c_i$ is inserted at the $i^{th}$ entry. This is
  clearly a dg-algebra homomorphism because $H_{-*}(\Omega G)$ is a
  polynomial algebra by hypothesis. Composing this $\Phi^{[n]}$ gives
  us an $A_\infty$-morphism:
  \[ H_{-*} (\Omega G) \to C_{-*}(\Omega G)^{\otimes n} \to C_{-*}(\Omega G) \]
  which is a quasi-isomorphism since it sends $y_i$ to $c_i$.
\end{proof} 

\begin{proof}[Proof of Theorem \ref{thm:formal}]
  First, we observe that we can restrict our attention to
  simply-connected Lie groups. Indeed, if $G$ is any compact,
  connected, simple Lie group, there exists a simply-connected finite
  cover $\pi: \tilde{G} \to G$ such that $G$ is the quotient of
  $\tilde{G}$ by a subgroup of the centre of $\tilde{G}$. Hence, for
  $p$ not dividing the order of $H_1(G)$, the covering map is a
  homotopy equivalence mod $p$.
	
  Working over a field $\kk$ of characteristic $p$, let us first
  observe that since $B\Omega G \cong G$, we have
  \[ C^*(G)=C^*(B\Omega G)= \mathrm{B}(C_{-*}(\Omega G))^\# \]
  where $\mathrm{B}$ is the bar complex and $\#$ denotes the
  $\kk$-linear dual. Now, this is the standard bar resolution that
  computes $\mathrm{Ext}_{C_{-*}(\Omega G)}(\kk,\kk)$, thus we have
  the Eilenberg-Moore equivalence of dg-algebras:
  \[ \mathrm{Rhom}_{C_{-*}(\Omega G)} (\kk , \kk) \cong C^*(G) \]
  On the other hand, $H_{-*}(\Omega G)$ is a symmetric algebra hence
  is formal by Proposition \ref{shm}. Therefore, we have:
  \[ \mathrm{Rhom}_{H_{-*}(\Omega G)} (\kk , \kk) \cong C^*(G) \]
  The left hand side has a bigrading given by the cohomological
  grading and the internal grading with respect to which $d$ has
  degree $(1,0)$. Applying the homological perturbation lemma, we get
  a quasi-isomorphic $A_\infty$-algebra on the homology such that the
  products $\mu^k: H^*(G)^{\otimes k} \to H^*(G)$ for $k \geq 2$ has
  bidegree $(2-k,0)$.

  On the other hand, by Koszul duality between the symmetric algebra
  $H_{-*}(\Omega G)$ and the exterior algebra $H^*(G)$
  (cf. \cite{bgs}), we know that the bigrading collapses to a single
  grading at the level of homology $H^*(G) \cong
  \mathrm{Ext}_{H_{-*}(\Omega G)} (\kk, \kk)$.  Hence, for grading
  reasons all the higher products $\mu^k$ for $k>2$ have to vanish. It
  follows that $C^*(G)$ is formal as required.
\end{proof}

\begin{rmk}
  It is an interesting question whether the formality result from
  Theorem \ref{thm:formal} holds in characteristic $p$ for which
  $H^*(G)$ has $p$-torsion. We note that $SO(3) \cong \rp{3}$ can be
  shown to be formal mod $2$ similar to the argument given in Example
  \ref{spheres}. Namely, recall that $H^*(\rp{3};\FF_2)=
  \FF_2[x]/(x^4)$ where $|x|=1$. Using a simplicial decomposition of
  $\rp{3}$, we can consider the normalized simplicial cochain complex
  $C^*(\rp{3}; \FF_2)$ which is non-zero only over $*=0,1,2,3$ and can
  be equipped with a product structure by dualizing the
  Alexander-Whitney diagonal approximation.  Therefore, we can
  construct a dg-algebra map $\FF_2[x]/(x^4)=H^*(\rp{3}; \FF_2) \to
  C^*(\rp{3}; \FF_2)$ by sending $x$ to a cochain level generator and
  extending it to a dg-algebra map. By construction, this is a
  quasi-isomorphism, thus we deduce that $C^*(\rp{3}; \FF_2)$ is
  formal.
\end{rmk}

\section{Twisted complexes}\label{sct:twcpx}

\subsection{Cones and nilpotence}

We will use the conventions of {\cite[Section (3p)]{Seidel}}. In
particular, the shift functor, which we denote by $[1]$, will conceal
a multitude of signs. Note that, with these conventions, if
$x\in\hom^n(X,Y)$ then $x[k]\in\hom^n(X[k],Y[k])$. Recall that, given
a closed morphism $\alpha\in\hom^0(X,Y)$, the cone on $\alpha$ is
represented by the twisted complex
\[X[1]\xdashrightarrow{-\alpha[1]}Y.\]
To avoid confusion, we will use solid arrows to denote morphisms of
degree zero and dashed arrows to denote morphisms of degree one.

Recall from {\cite[Section 3l]{Seidel}} that a twisted complex
$(\bigoplus_iL_i[s_i],\delta=\{\delta_{ij}\})$ is a direct sum of
shifts of objects $L_i[s_i]$ together with a collection of morphisms
$\delta_{ij}\colon L_i[s_i]\to L_j[s_j]$ of degree one satisfying some
conditions (a lower-triangularity condition and a closedness
condition {\cite[Equation 3.19]{Seidel}}). When we draw a collection
of objects and (dashed, degree one) arrows and call it a twisted
complex, this is shorthand for the twisted complex obtained by summing
the objects and taking the indicated morphisms as $\delta$. For
example, the twisted complex
\[
X[1]\xdashrightarrow{-\alpha[1]}Y\xdashleftarrow{-1[1]}Y[1]\xdashrightarrow{-\beta[1]}Z
\]
means we take $(\bigoplus_iL_i[s_i],\delta)$ with
\[L_1[s_1]=X[1],\ L_2[s_2]=Y[1],\ L_3[s_3]=Y,\ L_4[s_4]=Z\]
and
\[\delta_{13}=-\alpha[1],\ \delta_{23}=-1[1],\ \delta_{24}=-\beta[1]\]
(all other $\delta_{ij}=0$). Despite the fact that some of the arrows
in the diagram go backwards, this is still lower-triangular
($\delta_{ij}=0$ for $j<i$) for the labelling of objects we have
picked.

A {\em morphism of twisted complexes} from $(\bigoplus_i
L_i[s_i],\delta)$ to $(\bigoplus_jM_j[t_j],\epsilon)$ comprises a
collection of morphisms $L_i[s_i]\to M_j[t_j]$. There is a
differential $\mu^1_{\Tw}$ and a composition $\mu^2_{\Tw}$ (indeed, an
$A_\infty$-structure $\mu^d_{\Tw}$) on the space of morphisms defined
by {\cite[Equation 3.20]{Seidel}}. For clarity, we will always write
our twisted complexes (and their differentials) horizontally and
morphisms between twisted complexes vertically downwards. With this
convention, the operations $\mu^d_{\Tw}$ are defined by stacking
morphisms vertically and summing over all possible
$A_\infty$-compositions of morphisms (including the internal
differentials of the twisted complexes) with suitable signs.

\begin{lma}\label{lma:comp}
Let $X,Y,Z$ be objects of a strictly unital $A_\infty$-category and
\[\alpha\in\hom^0(X,Y),\quad\beta\in\hom^0(Y,Z)\]
be two $\mu_{\Tw}^1$-closed morphisms. Then the twisted complex
\begin{equation}\label{eq:conecomp}
X[1]\xdashrightarrow{-\alpha[1]}Y\xdashleftarrow{-1[1]}Y[1]\xdashrightarrow{-\beta[1]}Z
\end{equation}
is quasi-isomorphic to the twisted complex
\begin{equation}\label{eq:compcone}
\Cone{\mu^2(\beta,\alpha)}=\left(X[1]\xdashrightarrow{-\mu^2(\beta,\alpha)[1]}Z\right).
\end{equation}
Since the twisted complex in Equation \eqref{eq:conecomp} is the cone
on a $\mu^1_{\Tw}$-closed morphism
\[\Cone{\beta}\to\Cone{\alpha},\]
this implies that $\Cone{\beta}$ and $\Cone{\alpha}$ generate
$\Cone{\mu^2(\beta,\alpha)}$.
\end{lma}
\begin{proof}
A quasi-isomorphism is given by
\begin{center}
\begin{tikzpicture}[->,>=stealth',shorten >=1pt,auto,node distance=1.8cm,
  thick]

  \node (1) {$X[1]$};
  \node (2) [right= of 1] {$Y$};
  \node (3) [right= of 2] {$Y[1]$};
  \node (4) [right= of 3] {$Z$};
  \node (5) [below= of 1] {$X[1]$};
  \node (6) [below= of 4] {$Z$};

  \path[dashed]
    (1) edge node [above] {$-\alpha[1]$} (2)
    (3) edge node [above] {$-1[1]$} (2)
    (3) edge node [above] {$-\beta[1]$} (4)
    (5) edge node [below] {$-\mu^2(\beta,\alpha)[1]$} (6);

  \path
    (2) edge node [below left] {$-\beta$} (6)
    (4) edge node [right] {$1$} (6)
    (1) edge node [left] {$1[1]$} (5);
\end{tikzpicture}\end{center}

with quasi-inverse

\begin{center}
\begin{tikzpicture}[->,>=stealth',shorten >=1pt,auto,node distance=1.8cm,
  thick]

  \node (1) {$X[1]$};
  \node (2) [right= of 1] {$Y$};
  \node (3) [right= of 2] {$Y[1]$};
  \node (4) [right= of 3] {$Z$};
  \node (5) [above= of 1] {$X[1]$};
  \node (6) [above= of 4] {$Z$};

  \path[dashed]
    (1) edge node [below] {$-\alpha[1]$} (2)
    (3) edge node [below] {$-1[1]$} (2)
    (3) edge node [below] {$-\beta[1]$} (4)
    (5) edge node [above] {$-\mu^2(\beta,\alpha)[1]$} (6);

  \path
    (6) edge node [right] {$1$} (4)
    (5) edge node [left] {$1[1]$} (1)
    (5) edge node [above right] {$-\alpha[1]$} (3);
\end{tikzpicture}\end{center}
\end{proof}

\begin{dfn}\label{dfn:nilp}
Given a morphism $a\in\hom^0(X[s],X)$, we set $a^1:=a$ and define
$a^k\in\hom^0(X[ks],X)$, $k\geq 1$, inductively by
$a^{k+1}:=\mu^2(a,a^k[s])$. For simplicity of notation we will define
\[a^{m_1\odot m_2\odot\cdots\odot m_n}:=(\ldots((a^{m_1})^{m_2})\ldots)^{m_n}\]
As we are working with $A_\infty$-algebras, which are not associative
unless $\mu^1=0$, the precise bracketing we have chosen matters. We
will say that $a$ is {\em nilpotent} if there exists a sequence
$m_1,\ldots,m_n$ such that $a^{m_1\odot m_2\odot\cdots\odot m_n}=0$.
\end{dfn}

\begin{cor}\label{cor:splgen}
For any morphism $a\in\hom^0(X[s],X)$ and any sequence of integers
$m_1,\ldots,m_n$, $\Cone{a}$ generates $\Cone{a^{m_1\odot m_2
    \odot\cdots\odot m_n}}$. In particular, if $a$ is nilpotent then
$\Cone{a}$ split-generates $X$.
\end{cor}
\begin{proof}
Lemma \ref{lma:comp} tells us that $\Cone{a}$ and $\Cone{a^k}$
generate $\Cone{a^{k+1}}$. Inductively, we see that $\Cone{a}$
generates $\Cone{a^{m_1}}$ for any $m_1\geq 1$. Applying this argument
again with $a$ replaced by $a^{m_1\odot m_2\odot\cdots\odot m_{n-1}}$,
we see that $\Cone{a}$ and $\Cone{a^{m_1\odot m_2\odot\cdots\odot
    m_{n-1}}}$ generate $\Cone{a^{m_1\odot m_2\odot\cdots\odot
    m_n}}$. By induction, $\Cone{a}$ generates $\Cone{a^{m_1\odot
    m_2\odot\cdots\odot m_n}}$. If $a^{m_1\odot m_2\odot\cdots\odot
  m_n}=0$ then
\[\Cone{a^{m_1\odot m_2\odot\cdots\odot m_n}}=X\oplus X[sm].\]
Therefore, if $a$ is nilpotent we see that $\Cone{a}$ split-generates
$X$.
\end{proof}

The following lemma will be useful for proving nilpotence of certain
morphisms.

\begin{lma}\label{lma:koszulinduct}
Suppose that $P$ is an object in an $A_\infty$-category and that
$x\in\hom^0(P[s],P)$ is a closed morphism. Suppose moreover that
$a\in\hom^0(P[t],P)$ and $c\in\hom^0(P[s+t+1],P)$ are such that
\begin{center}
\begin{tikzpicture}[->,>=stealth',shorten >=1pt,auto,node distance=2cm and 2.5cm,
  thick]

  \node at (0,0) (1) {$P[s+t+1]$};
  \node at (0,-3) (2) {$P[s+1]$};
  \node at (4,0) (3) {$P[t]$};
  \node at (4,-3) (4) {$P$};

  \path
    (1) edge node [left] {$a[s+1]$} (2)
    (3) edge node [right] {$a$} (4)
    (1) edge node [below left] {$c$} (4);

  \path[dashed]
    (1) edge node [above] {$-x[t+1]$} (3)
    (2) edge node [below] {$-x[1]$} (4);
\end{tikzpicture}\end{center}
defines a $\mu^1_{\Tw}$-closed morphism
$b\colon\Cone{\TO{P[s]}{x}{P}}[t]\to\Cone{\TO{P[s]}{x}{P}}$. Then
$b^k$ has the form
\begin{center}
\begin{tikzpicture}[->,>=stealth',shorten >=1pt,auto,node distance=2cm and 2.5cm,
  thick]

  \node at (0,0) (1) {$P[s+kt+1]$}; \node at (0,-3) (2) {$P[s+1]$};
  \node at (4,0) (3) {$P[kt]$}; \node at (4,-3) (4) {$P$};

  \path (1) edge node [left] {$a^k[s+1]$} (2) (3) edge node [right]
        {$a^k$} (4) (1) edge node [below left] {$c'$} (4);

  \path[dashed] (1) edge node [above] {$-x[kt+1]$} (3) (2) edge node
       [below] {$-x[1]$} (4);
\end{tikzpicture}\end{center}
for some $c'\in\hom^0(P[s+kt+1],P)$. Moreover, if $a^m=0$ then
$(b^m)^2=0$.
\end{lma}
\begin{proof}
  We will verify the claim about the form of $b^k$ by induction. It is
  clearly true for $k=1$. We need to check that $b^{k+1}$ has no
  component connecting $P[(k+1)t]$ to $P[s+1]$. By definition of
  $\mu^2_\Tw$ (the composition for morphisms of twisted complexes
  {\cite[Equation 3.20]{Seidel}}) this component is a sum of terms of
  the form $\mu^m(a,\delta_{m-1},\ldots,\delta_2,a^k)$ where
  $\delta_i$ stands for some $\delta\in\hom^0(P[t],P[s+t+1])$ occuring
  as part of the differential in the twisted complex
  $\Cone{x[t]}$. However, the differential in $\Cone{x[t]}$ has no
  component connecting $P[t]$ and $P[s+t+1]$ (this would point left in
  our picture, where the only differential is the right-pointing arrow
  $-x[t+1]$).

  We now show that $a^m=0$ implies $(b^m)^2=0$. If $a^m=0$ then the
  first part of the lemma tells us that $b^m$ has the form

\begin{center}
\begin{tikzpicture}[->,>=stealth',shorten >=1pt,auto,node distance=2cm and 2.5cm,
  thick]

  \node at (0,0) (1) {$P[s+mt+1]$}; \node at (0,-3) (2) {$P[s+1]$};
  \node at (4,0) (3) {$P[mt]$}; \node at (4,-3) (4) {$P$};

  \path (1) edge node [left] {$0$} (2)
  (3) edge node [right] {$0$} (4)
  (1) edge node [below left] {$c'$} (4);

  \path[dashed] (1) edge node [above] {$-x[mt+1]$} (3)
  (2) edge node [below] {$-x[1]$} (4);
\end{tikzpicture}\end{center}

We find $(b^m)^2$ (that is $\mu^2_{\Tw}(b^m,b^m)$ by stacking $b^m$ on
top of itself and summing over all possible compositions:

\begin{center}
\begin{tikzpicture}[->,>=stealth',shorten >=1pt,auto,node distance=2cm and 2.5cm,
  thick]

  \node at (0,0) (5) {$P[s+2mt+1]$}; \node at (4,0) (6) {$P[2mt]$};
  \node at (0,-3) (1) {$P[s+mt+1]$}; \node at (0,-6) (2) {$P[s+1]$};
  \node at (4,-3) (3) {$P[mt]$}; \node at (4,-6) (4) {$P$};

  \path (1) edge node [left] {$0$} (2) (3) edge node [right]
        {$0$} (4) (1) edge node [below left] {$c'$} (4);

  \path[dashed] (1) edge node [above] {$-x[mt+1]$} (3) (2) edge node
       [below] {$-x[1]$} (4);

  \path (5) edge node [left] {$0$} (1) (6) edge node [right]
        {$0$} (3) (5) edge node [below left] {$c'$} (3);

  \path[dashed] (5) edge node [above] {$-x[2mt+1]$} (6);
\end{tikzpicture}\end{center}

In this picture there are no non-zero compositions because no non-zero
arrow ends where another begins, therefore $(b^m)^2=0$.
\end{proof}

\subsection{Koszul twisted complexes}

We are interested in a certain class of iterated cones, a class which
we will see is preserved by triangulated $A_\infty$-functors.

\begin{dfn}\label{dfn:koszul}
Let $\aA$ be an $A_\infty$-category and let $K$ be a twisted complex
in $\Tw(\aA)$. We say $K$ is a {\em Koszul twisted complex built out
  of $K'$} if the following conditions hold.
\begin{enumerate}
\item[I.] $K$ can be expressed as an iterated cone of the form:
\begin{align*}
K=K_n&=\Cone{\TO{K_{n-1}[s_{n-1}]}{x_{n-1}}{K_{n-1}}}\\
  K_{n-1}&=\Cone{\TO{K_{n-2}[s_{n-2}]}{x_{n-2}}{K_{n-2}}}\\
  \vdots&\quad\vdots\quad\vdots\\
  K_1&=\Cone{\TO{K_0[s_0]}{x_0}{K_0}}\\
  K_0&=K'
\end{align*}
for some sequence of objects $K_0,\ldots,K_{n-1}$, shifts
$s_0,\ldots,s_{n-1}$ and degree zero, $\mu^1_{\Tw}$-closed morphisms
$x_0,\ldots,x_{n-1}$.
\item[II.] Moreover, for $0 \leq i \leq n-1$, each $x_i$ must be expressible via the following recursive construction. For each $i$, there exist morphisms
  $a_{i,j}\in\hom^0\left(K_j[s_i],K_j\right)$ and
  $c_{i,j}\in\hom^0\left(K_j[s_i+s_j+1],K_j\right)$ for $0 \leq j <i$, such that
  $\mu^1_{\Tw}(a_{i,j})=0$ and the following diagram
\begin{center}
\begin{tikzpicture}[->,>=stealth',shorten >=1pt,auto,thick,node distance=2cm and 3cm]

  \node at (0,0) (1) {$K_j[s_i+s_j+1]$};
  \node at (0,-3) (2) {$K_j[s_j+1]$};
  \node at (4,0) (3) {$K_j[s_i]$};
  \node at (4,-3) (4) {$K_j$};

  \path
    (1) edge node [left] {$a_{i,j}[s_j+1]$} (2)
    (3) edge node [right] {$a_{i,j}$} (4)
    (1) edge node [below left] {$c_{i,j}$} (4);

  \path[dashed]
    (1) edge node [above] {$-x_j[s_i+1]$} (3)
    (2) edge node [below] {$-x_j[1]$} (4);
\end{tikzpicture}\end{center}
	viewed as a map $\Cone{x_j}[s_i] \to \Cone{x_j}$ gives $a_{i,j+1}$ for $0\leq j < i-1$, and finally it gives $x_i$ for $j=i-1$.
\end{enumerate}
If $K$ is a Koszul twisted complex as above, we call the morphisms
$a_{i,0}$, $i=0,\ldots,n-1$, its {\em edges}.
\end{dfn}

\begin{lma}\label{lma:functorkoszul}
Suppose that $K$ is a Koszul twisted complex in $\Tw(\aA)$ and that $\fF\colon\Tw(\aA)\to\Tw(\bB)$
is the triangulated $A_\infty$-functor $\Tw(\gG)$ induced by an $A_\infty$-functor $\gG\colon\aA\to\bB$. Then $\fF(K)$ is a Koszul twisted complex in $\Tw(\bB)$.
\end{lma}
\begin{proof}
By {\cite[Lem. 3.30]{Seidel}}, we have $\Cone{\fF^1(m)}=\fF(\Cone{m})$
for any morphism $m$, so we only need to check that Condition II of
the definition of Koszul twisted complexes remains true for
$\fF^1(a_{i,j+1})$. The way we define $\fF$ for morphisms of twisted
complexes means that applying $\fF$ to the diagram in Condition II
yields

\begin{center}
\begin{tikzpicture}[->,>=stealth',shorten >=1pt,auto,node distance=2cm and 3cm,
  thick]

  \node at (0,0) (1) {$\fF(K_j)[s_i+s_j+1]$};
  \node at (0,-3) (2) {$\fF(K_j)[s_j+1]$};
  \node at (6,0) (3) {$\fF(K_j)[s_i]$};
  \node at (6,-3) (4) {$\fF(K_j)$};

  \path
    (1) edge node [left] {$\fF^1(a_{i,j})[s_j+1]$} (2)
    (3) edge node [right] {$\fF^1(a_{i,j})$} (4)
    (1) edge node [below left] {$C$} (4);

  \path[dashed]
    (1) edge node [above] {$-\fF^1(x_j)[s_i+1]$} (3)
    (2) edge node [below] {$-\fF^1(x_j)[1]$} (4);
\end{tikzpicture}\end{center}

where
\[C=\fF^1(c_{i,j})\pm\fF^2(a_{i,j},x_j[s_i+1])\pm\fF^2(x_j[1],a_{i,j}[s_j+1]).\]
This diagram now represents $\fF^1(a_{i,j+1})$ and this has the
required form.
\end{proof}

\subsection{Extreme cases}

There are two extreme cases which are useful to consider: the
case when the edges of a Koszul twisted complex are nilpotent, and the
case when one of the morphisms $a_{i,j}$ is a quasi-isomorphism. In
the first case, we will see that $K$ split-generates $K'$. In the
second case, we will see that $K$ is quasi-isomorphic to the zero
object.

\begin{lma}\label{lma:nilpinduct}
Suppose $K$ is a Koszul twisted complex built out of $K'$, for which
$(a_{i,j})^{m\odot 2^{\odot q}}=0$ for some $i,j$. Then
$(a_{i,j+1})^{m\odot 2^{\odot(q+1)}}=0$. In particular, if
$(a_{i,0})^m=0$ then $(a_{i,i})^{m\odot 2^{\odot i}}=0$.
\end{lma}
\begin{proof}
This is immediate from Lemma \ref{lma:koszulinduct}.
\end{proof}

\begin{cor}\label{cor:nilpgen}
If $K$ is a Koszul twisted complex built out of $K'$, such that there
exist integers $m_0,\ldots,m_{n-1}$ for which $(a_{i,0})^{m_i}=0$ for
each $i=0,1,\ldots,n-1$ then $K$ split generates $K'$.
\end{cor}
\begin{proof}
Lemma \ref{lma:nilpinduct} tells us that $x_i=a_{i,i}$ is
nilpotent. Corollary \ref{cor:splgen} then implies that $K_{i+1}$
split-generates $K_i$. Since this holds for each $i$, this implies
that $K=K_n$ split-generates $K_0=K'$.
\end{proof}

\begin{lma}\label{lma:quisoinduct}
Suppose that $K$ is a Koszul twisted complex in which $a_{i,j}$ is a
quasi-isomorphism for some $0\leq j\leq i\leq n-1$. Then $a_{i,j+1}$
is also a quasi-isomorphism
\end{lma}
\begin{proof}
Let $J$ be an object of the given $A_\infty$-category. We will show
that the induced map
\[\TO{H(\hom(J,K_{j+1}[s_i]))}{(a_{i,j+1})_*}{H(\hom(J,K_{j+1}))}\]
is an isomorphism. Since $K$ is a Koszul twisted complex,
$\hom(J,K_{j+1})$ has the form
\[\left(\hom(J,K_j[1])\oplus\hom(J,K_j[s_j]),
   \left(\begin{array}{cc}
      \partial_{J,K_j}[1] & 0\\
      -x_j[1]          &\partial_{J,K_j}[s_i]
   \end{array}\right)\right)\]
This differential is lower triangular, hence there is a two-step
spectral sequence whose $E_1$-page is $H(\hom(J,K_j[1]))\oplus
H(\hom(J,K_j[s_i]))$ and which converges to $H(\hom(J,K_{j+1}))$. The
map $\left(a_{i,j+1}\right)_*$ has the form
\[   \left(\begin{array}{cc}
      a_{i,j}[s_j+1] & 0\\
      c_{i,j}     & a_{i,j}
   \end{array}\right)\]
and therefore induces a map of filtered complexes between
$\hom(J,K_j[s_i])$ and $\hom(J,K_j)$. This map is an isomorphism on
the $E_1$-page since $a_{i,j}$ is a quasi-isomorphism. Therefore
$(a_{i,j+1})_*$ is an isomorphism as desired. A similar argument
proves that
\[\TO{H(\hom(K_{j+1},J))}{(a_{i,j+1})^*}{H(\hom(K_{j+1}[s_i],J))}\]
is an isomorphism for any test object $J$.
\end{proof}

\begin{cor}\label{cor:quisoimplieszero}
Suppose that $K$ is a Koszul twisted complex such that $a_{i,0}$ is a
quasi-isomorphism for some $i$. Then $K$ is quasi-isomorphic to zero.
\end{cor}
\begin{proof}
By induction from Lemma \ref{lma:quisoinduct}, it follows that $a_{i,j}$
is a quasi-isomorphism for all $0\leq j\leq i$. In particular,
$x_i=a_{i,i}$ is a quasi-isomorphism, hence its cone, $K_{i+1}$, is
quasi-isomorphic to zero. Since $K=K_n$ is generated by $K_{i+1}$ (by
definition), $K$ must also be quasi-isomorphic to zero.
\end{proof}

\subsection{Koszul twisted complexes in cotangent bundles of groups}

Let $\kk$ be a field and let $e_1,\ldots,e_m, e_{m+1},\ldots ,e_{m+n}$
be a collection of graded variables with gradings in negative even
degrees $-s_1,\ldots,-s_{m}$ and $s_{m+1}=\ldots = s_{m+n}=0$. Let $A$
denote the $\ZZ$-graded algebra
\[ \kk[e_1,\ldots,e_m] \otimes \kk[ e_{m+1}^{\pm} , \ldots, e_{m+n}^{\pm}] \]
generated by these variables, considered as a $\ZZ$-graded
$A_\infty$-algebra with vanishing higher products. The algebra $A$ is
the prototype of $H_{-*}(\Omega G)$ for $G$ a compact connected Lie
group that is a product of a simple Lie group and a torus.

Let $\bimod{A}$ denote the category of $A_\infty$-bimodules over
$A$. We summarise some basic results about Koszul complexes (in the
usual sense).

\begin{lma}\label{lma:koszulbasic}
\begin{enumerate}
\item[(a)] There is a unique $A_\infty$ $A$-bimodule structure on the
  field $\kk$, where the $e_i$ for $i=1,\ldots m$ and $1-e_i$ for $i=m+1,\ldots m+n$ annihilate the module and the higher
  $A_\infty$-operations are zero (for grading reasons).
\item[(b)] Let $K(e_i)$ denote the complex
  $\TODASH{A[-s_i+1]}{-e_i[1]}{A}$ for $i=1,\ldots, m$ and the complex
  $\TODASH{A[1]}{-1[1]+e_i[1]}{A}$ for $i=m+1,\ldots m+n$, and let
  $K(\underline{e})$ be the tensor product $K(e_1)\otimes\cdots\otimes
  K(e_{m+n})$. Let $\kk$ be the ground field considered as an
  $A_\infty$-bimodule with trivial higher products. The complex
  $K(\underline{e})$ is the standard Koszul resolution of $\kk$.
\end{enumerate}
\end{lma}

\begin{lma}\label{lma:grps}
In the $A_\infty$-category $\Tw(\bimod{A})$, the object
$K(\underline{e})$ is a Koszul twisted complex built out of $A$, with
	shifts $s_i$, edges $a_{i,0}=e_i$ for $i=1,\ldots, m$ and $a_{i,0}=1-e_i$ for $i=m+1,\ldots m+n$, and diagonal morphisms $c_{i,j}=0$.
\end{lma}
\begin{proof}
Let $C$ be a complex of ordinary $A$-modules let $e_i\star$ denote the
left module action of $e_i\in A$. The total tensor product complex for
$C\otimes K(e_i)$ exhibits $C\otimes K(e_i)$ as a cone on the morphism
$e_i\star\colon C[s_i]\to C$. If $C$ is itself a cone
$\Cone{\TO{C_1}{\gamma}{C_2}}$ then this module action has the form
\begin{center}
\begin{tikzpicture}[->,>=stealth',shorten >=1pt,auto,node distance=2cm and 2.5cm,
  thick]

  \node at (0,0) (1) {$C_1[s_i+1]$};
  \node at (0,-3) (2) {$C_1[1]$};
  \node at (4,0) (3) {$C_2[s_i]$};
  \node at (4,-3) (4) {$C_2$};

  \path
    (1) edge node [left] {$(e_i\star)[1]$} (2)
    (3) edge node [right] {$e_i\star$} (4)
    (1) edge node [below left] {$0$} (4);

  \path[dashed]
    (1) edge node [above] {$-\gamma[s_i+1]$} (3)
    (2) edge node [below] {$-\gamma[1]$} (4);
\end{tikzpicture}\end{center}
If we take $C_k=K(e_1)\otimes\cdots\otimes K(e_k)$ then we deduce that $C_{k+1}=\Cone{e_{k+1}\star}$.
\end{proof}

\begin{cor}\label{cor:grps}
Let $G$ be a compact, connected Lie group and let $\kk$ be a field of
characteristic $p\geq 0$ such that $p$ is not a torsion prime for
$G$. Then, in the wrapped Fukaya category of $T^*G$, the zero-section
is quasi-isomorphic to a Koszul twisted complex built out of $T^*_1G$.
\end{cor}
	\begin{proof} As usual, by passing to a finite cover (which is a homotopy equivalence away from torsion primes for $G$), we can restrict our attention to the case of a product of a simple non-abelian Lie group and a torus. By Theorem \ref{thm:formal}, the dga $C_{-*}(\Omega G;\kk)$ of cubical
chains on the based loop space of $G$, equipped with the Pontryagin
product, is quasi-isomorphic to its homology, $H_{-*}(\Omega
G;\kk)$. This homology group is tensor product of a polynomial ring on generators of negative even degree and a Laurent polynomial ring in degree 0 so Lemma \ref{lma:grps} applies.

There is a full and faithful embedding of the wrapped Fukaya category
of $T^*G$ into $\bimod{C_{-*}(\Omega G;\kk)}$ \cite{AbouzaidCotFib}
which takes a Lagrangian brane to its Floer $A_\infty$-bimodule with
the cotangent fibre. Under this embedding, the $A_\infty$-bimodule
$\kk$ represents the zero-section, since the zero-section and the
cotangent fibre intersect at precisely one point. Lemma \ref{lma:grps}
tells us that there is a Koszul twisted complex $K$ quasi-isomorphic
to $\kk$, built out of $C_{-*}(\Omega G;\kk)$. Since the wrapped
Fukaya category maps fully faithfully into the bimodule category, this
tells us that there is a Koszul twisted complex in $\wW(T^*G)$
quasi-isomorphic to the zero-section, built out of $T^*_1G$.
\end{proof}

\subsection{Twisted complexes of functors}\label{sct:twcpx-functors}

Recall that an $A_{\infty}$-functor
$\mathcal{F}\colon\mathcal{A}\to\bB$ comprises an assignment of
objects $\mathcal{F}X\in\bB$ to every object $X\in\mathcal{A}$ and a
sequence of multilinear maps
\[\mathcal{F}^k\colon\hom_\aA(X_k,X_{k+1})\otimes\cdots\otimes\hom_\aA(X_1,X_2)\to\hom_{\bB}(FX_1,FX_2)[1-k]\]
satisfying a sequence of equations (see {\cite[Eq
    (1.6)]{Seidel}}). The non-unital $A_{\infty}$-functors
$\mathcal{A}\to\bB$ form an $A_{\infty}$-category
$\nufun(\mathcal{A},\bB)$. Given $F,G\in \nufun(\mathcal{A},\bB)$, an
element $T\in\hom_{\nufun}^g(F,G)$ is given by a {\em pre-natural
  transformation} $T^0,T^1,T^2,\ldots$ where $T^0$ assigns an element
$T^0_X\in\hom^g_{\bB}(FX,GX)$ for each $X$ in $\mathcal{A}$ and where
$T^k\colon\hom(X_k,X_{k+1})\otimes\cdots\otimes\hom(X_1,X_2)\to\hom(X_1,X_{k+1})[g-k]$
are multilinear maps.

There is an $A_{\infty}$-structure $\lambda$ on $\nufun(\aA,\bB)$,
described in {\cite[Eq. 1.9]{Seidel}}. The following property of this
$A_{\infty}$-structure will be important to us: if
$S,T\in\hom_{\nufun}(\OP{Id}_{\aA},\OP{Id}_{\aA})$ then we have
\begin{equation}\label{eq:lambdamu}
\left(\lambda^k(T_k,\ldots,T_1)\right)^0_X=\mu_{\aA}^k((T_k^0)_X,\ldots,(T_1^0)_X).
\end{equation}

Let $(\bigoplus_iF_i,\{T_{ij}\})$ be a twisted complex of functors in
$\Tw(\nufun(\aA,\bB))$. Using Equation \eqref{eq:lambdamu}, the
Maurer-Cartan equation
\[\sum\lambda^k(T,\ldots,T)=0\]
implies the Maurer-Cartan equation
\[\sum\mu^k(T^0_X,\ldots,T^0_X)=0\]
for the twisted complex $(\bigoplus_iF_iX,\{(T_{ij}^0)_X\})$ in
$\bB$. The following lemma is easy to verify.

\begin{lma}\label{lma:twcpxfunctors}
Let $\aA$ and $\bB$ be strictly unital
$A_{\infty}$-categories. Suppose that $F=(\bigoplus_iF_i,\{S_{ij}\})$
and $G=(\bigoplus_iG_i,\{T_{ij}\})$ are quasi-isomorphic twisted
complexes in $\Tw\left(\nufun(\aA,\bB)\right)$. Then for any
$X\in\aA$, the objects $(\bigoplus_iF_iX,\{S_{ij}^0X\})$ and
$(\bigoplus_iG_iX,\{T_{ij}^0X\})$ are quasi-isomorphic in $\Tw(\bB)$.
\end{lma}

\section{Quantum cohomology and generation of the diagonal}\label{sct:qh}

\subsection{Block decomposition of quantum cohomology}

The quantum cohomology of a monotone symplectic manifold decomposes as
a direct sum of local rings. It is not special to quantum cohomology,
and the proof cited works for any (super) commutative $\kk$-algebra, finite-dimensional over $\kk$.

\begin{lma}\label{lma:qhnilp}
Let $X$ be a compact, monotone symplectic manifold and $\kk$ be an arbitrary field. Then there exists a finite set $W$ and
idempotent elements
\[\left\{e_\alpha\in \QH^{ev}(X;\kk)\right\}_{\alpha\in W}\]
such that $1=\sum_{\alpha\in W}e_\alpha$ and $e_\alpha e_\beta=0$ if
$\alpha\neq\beta$. Moreover, each summand
\[\QH(X;\kk)_\alpha=e_\alpha \QH(X;\kk)\]
is a subalgebra which is a local ring with nilpotent maximal ideal. In
other words, for every element $x\in \QH(X;\kk)_\alpha$, either
$x^m=0$ for some $m$ or else there exists $y\in \QH(X;\kk)_\alpha$
such that $xy=e_\alpha$.
\end{lma}
\begin{proof} A finite-dimensional algebra over a field is an Artinian ring. Now, a commutative Artinian ring is uniquely (up to isomorphism) a finite direct product of commutative Artininan local rings (see for example \cite[Theorem 8.7]{atiyahmacdonald}). The assertions follow easily from this.
\end{proof}

\begin{rmk}
  The property of a commutative ring splitting as a finite direct sum
  of local rings with nilpotent maximal ideal is equivalent to being
  the ring of functions on a zero-dimensional Noetherian scheme. In
  the case of quantum cohomology of a toric Fano variety, which is
  supported in even degree and hence commutative, the zero-dimensional
  scheme in question is the critical locus of the mirror
  superpotential, which consists of isolated critical points.
\end{rmk}

\begin{rmk}
  For each $\alpha$, let $\mathfrak{m}_\alpha$ be the maximal ideal of
  the local ring $QH(X;\kk)_\alpha$, then it follows by the
  nullstellensatz that the residue field $\kk_\alpha=
  QH(X;\kk)_{\alpha}/\mathfrak{m}_\alpha$ is a finite extension of
  $\kk$.
\end{rmk}

\subsection{Lagrangians in $X^{-}\times X$}

Recall that the quantum cohomology is isomorphic, via the PSS map, to
the self-Floer cohomology of the diagonal $\Delta\subset (X^{-} \times
X,(-\omega)\oplus\omega)$. It therefore admits a chain-level
description as an $A_\infty$-algebra $\CF(\Delta,\Delta;\kk)$. We can
transfer this $A_\infty$-structure to the quantum cohomology via
homological perturbation. Seidel {\cite[Lem. 4.2]{Seidel}} tells us
that the idempotents $e_\alpha$ in $\QH(X;\kk)$ occur as the first
part $\mathcal{E}_\alpha^1$ of an {\em idempotent up to homotopy},
$\mathcal{E}_\alpha$, defining an object in the Fukaya category of
$X^{-}\times X$ which behaves like a summand of $\Delta$ and which we
will write as $\Delta_\alpha$. Note that
\[\HF(\Delta_\alpha,\Delta_\beta;\kk)=e_\alpha \QH(X;\kk)e_\beta=
\begin{cases}
\QH(X;\kk)_\alpha&\mbox{ if }\alpha=\beta\\
0&\mbox{ otherwise.}
\end{cases}\]
Let $L\subset X$ be a Lagrangian submanifold. Under the zeroth
closed-open map $CO^0\colon \QH(X;\kk)\to \HF(L,L;\kk)$, the
idempotent $e_\alpha$ is sent to an idempotent (possibly zero) in
$\HF(L,L;\kk)$. Again, this lifts to an idempotent up to homotopy and
defines an object $L_\alpha$ in the Fukaya category
$D^{\pi}\fF(X;\kk)$. Since the identity element splits as a sum of
orthogonal idempotents, the Fukaya category splits as a disjoint union
of subcategories $D^{\pi}\fF(X;\kk)_\alpha$ comprising precisely those
summands in which $CO^0(e_\alpha)$ hits the identity element.

\begin{rmk} \label{boilsdown}
We emphasise that $L_\alpha\neq 0$ if and only if $CO^0(e_\alpha)\neq
0\in \HF(L,L;\kk)$. In Section \ref{sct:examples} below, much of the
challenge will be to decide, given a Lagrangian $L$ with
$\HF(L,L;\kk)\neq 0$, for which $\alpha$ is
$\HF(L_\alpha,L_\alpha;\kk)\neq 0$? By {\cite[Proposition
    1.2]{Sheridan}}, we have $CO^0(2c_1(X))=2\mathfrak{m}_0(L)1_L$,
so, in the happy circumstance that the local summands of $\QH(X;\kk)$
are generalised eigenspaces of $2c_1(X)$ with distinct eigenvalues,
one knows that $L$ belongs to the summand corresponding to the
eigenvalue $2\mathfrak{m}_0(L)$. In many interesting examples, the
local summands refine the splitting into generalised eigenspaces, so
one needs more information than $\mathfrak{m}_0(L)$ to decide for
which summands $\alpha$ we have $L_\alpha\neq 0$. See Remark
\ref{rmk:summands}.
\end{rmk}

\begin{lma}\label{lma:diaggen}
Suppose $M$ is an object in the Fukaya category of $X^{-}\times X$
which is quasi-isomorphic to a Koszul twisted complex built out of
$\Delta$. Then $M$ split-generates the subcategory
$D^{\pi}\fF(X;\kk)_\alpha$ if and only if $\HF(M,\Delta_\alpha)\neq
0$.
\end{lma}
\begin{proof}
Since $\Delta=\bigoplus_{\alpha\in W}\Delta_\alpha$ and the objects
$\Delta_\alpha$ are orthogonal in the Fukaya category, the Koszul
twisted complex representing the object $M$ splits as a direct sum
$\bigoplus M_\alpha$ of objects each of which is quasi-isomorphic to a
Koszul twisted complex $K_\alpha$ built out of $\Delta_\alpha$. We
will analyse these independently of one another, so for the remainder
of the proof we will fix $\alpha\in W$ and omit the
$\alpha$-decorations from the morphisms in the Koszul twisted complex.

By Lemma \ref{lma:qhnilp}, the morphisms $a_{i,0}$ in the Koszul
twisted complex $K_\alpha$ are all either nilpotent or invertible in
the ring $\QH(X;\kk)_\alpha$. In case one of them is invertible, it
gives a quasi-isomorphism $\Delta_\alpha\to\Delta_\alpha$ and hence,
by Corollary \ref{cor:quisoimplieszero}, $M_\alpha=0$. In case none of
the $a_{i,0}$ is invertible, they are all nilpotent, and Corollary
\ref{cor:nilpgen} implies that $M_\alpha$ split-generates
$\Delta_\alpha$. Therefore $M_\alpha$ split-generates $\Delta_\alpha$
if and only if $M_\alpha\neq 0$.

Since $M_\alpha$ is itself generated by $\Delta_\alpha$, and since
$\HF(M,\Delta_\alpha;\kk)=\HF(M_\alpha,\Delta_\alpha;\kk)$, we see
that $M_\alpha\neq 0$ if and only if $\HF(M,\Delta_\alpha;\kk)\neq 0$;
hence $M_\alpha$ split-generates $\Delta_\alpha$ if and only if
$\HF(M,\Delta_\alpha;\kk)\neq 0$.
\end{proof}

\section{Quilted Floer theory}\label{sct:quilt}

In Section \ref{sct:stdquilt}, we will briefly review the idea of
quilted Floer theory in the form it appears in \cite{MWW}. We will
need this theory to work in a slightly more general setting, involving
flat line bundles (Section \ref{sct:locsysquilt}) and some mildly
noncompact Lagrangians in exact manifolds with non-trivial fundamental
group (Section \ref{sct:wrappedquilt}).  In each case, we will
indicate what technical modifications need to be made. We expect that
the construction given here can be generalised vastly and carried out
in a more natural setting by appealing to the recent work of Fukaya
\cite{Fbetterquilt}.

\subsection{Standard theory}\label{sct:stdquilt}

Given a symplectic manifold $Z$ with symplectic form $\omega$, we will
use $Z^{-}$ to denote the same manifold equipped with the symplectic
form $-\omega$. A {\em Lagrangian correspondence} between $Y$ and $Z$
is a Lagrangian submanifold $C\subset Y^{-}\times Z$. We will write
$C^{-}$ for the same submanifold considered as a correspondence in
$Z^{-}\times Y$.

A {\em generalised Lagrangian correspondence from $Y$ to $Z$} is a
sequence $\underline{L}=\{L_{i,i+1}\subset A_i^{-}\times
A_{i+1}\}_{i=0}^{k-1}$ of Lagrangian correspondences with $A_0=Y$ and
$A_k=Z$. A brane structure on a generalised correspondence comprises a
choice of orientation, relative spin structure and grading on each
$L_{i,i+1}$ (the grading may only live in $\ZZ/2$). We denote by
$\underline{L}^T$ the generalised Lagrangian correspondence
$\left(L_{k-1,k}^-,\ldots,L_{01}^-\right)$, where the minus sign
denotes a reversal of orientation and associated brane data.

Given two generalised correspondences $\underline{K}$ from $Z_1$ to
$Z_2$ and $\underline{L}$ from $Z_2$ to $Z_3$, there is a concatenated
generalised correspondence $(\underline{K} , \underline{L}) $ from
$Z_1$ to $Z_3$. A generalised Lagrangian correspondence from $\{pt\}$
to $Z$ is called a {\em generalised Lagrangian submanifold} of $Z$.

Given
\begin{equation}\label{eq:gls}
\underline{K}=\{K_{i,i+1}\subset A_i^{-}\times A_{i+1}\}_{i=0}^{k-1}\qquad
\underline{L}=\{L_{i,i+1}\subset B_i^{-}\times B_{i+1}\}_{i=0}^{\ell-1}
\end{equation}
generalised Lagrangian submanifolds of $Z=A_k=B_\ell$, a {\em generalised
intersection point} is a tuple
\[(x_0,x_1,\ldots,x_{k+\ell})\in A_0\times\cdots\times A_{k-1}\times Z\times B_{\ell-1}\times\cdots\times B_0\]
such that $(x_i,x_{i+1})\in K_{i,i+1}$ for $i<k$ and $(x_i,x_{i+1})\in
L^-_{\ell-i,\ell-i+1}$ for $i\geq k$. Let us denote by
$\underline{K}\cap\underline{L}$ the set of generalised intersection
points of $(\underline{K},\underline{L}^T)$, which one can arrange to
be a finite set after a Hamiltonian perturbation.

Given $\underline{K}$ and $\underline{L}$ as in Equation
\eqref{eq:gls}, one defines the Floer cochain complex
$CF(\underline{K},\underline{L})$ to be
\[\bigoplus_{\underline{x}\in\underline{K}\cap\underline{L}}\kk\langle\underline{x}\rangle.\]
We explain the construction of the differential, which counts
pseudoholomorphic quilted strips. First let
$\underline{M}=(\underline{K}, \underline{L}^T) =\{M_{i,i+1}\subset
C_i^-\times C_{i+1}\}_{i=0}^{m-1}$ ($m=k+\ell$). The domain of a
pseudoholomorphic quilted strip is $\RR \times [0,1]$, but it has a
number of {\em seams} $\sigma_i=\RR \times \{p_i\} $,
$i=1,\ldots,m-2$.  The seams cut the domain into {\em patches}
$S_i=\RR \times [p_{i-1},p_i] $, $i=1,\ldots,m-1$, where we define
$p_0=0$ and $p_{m-1}=1$. A {\em quilted strip} consists of an
$(m-1)$-tuple $\underline{u}=(u_1,\ldots,u_{m-1})$ of maps $u_i\colon
S_i\to C_i$ for $i=1,\ldots,m-1$, satisfying the boundary and seam
conditions \[u_1(s,0)\in M_{01},\quad
(u_i(s,p_{i}),u_{i+1}(s,p_{i}))\in M_{i,i+1},\quad u_{m-1}(s,1)\in
M_{m-1,m}\] for all $s\in\RR$. A {\em pseudoholomorphic quilted strip}
is a finite-energy quilted strip such that each $u_i$ satisfies the
Floer equation with respect to a choice of translation invariant
time-dependent almost complex structure $J_i$ on $C_i$ and a suitable
choice of Hamiltonian perturbation.

More generally, one can define a {\em quilted Riemann surface}
{\cite[Definition 3.1]{WWpseudoquilt}}, which is a Riemann surface $S$
with strip-like ends, separated into patches $S_i$ by a collection of
oriented, properly embedded, disjoint real analytic arcs. We write
$\sigma_{ij}$ for an oriented seam which has $S_i$ on its right and
$S_j$ on its left; if $b$ is a component of $\partial S$ then we write
$S_{i(b)}$ for the patch containing $b$.

\begin{dfn}\label{dfn:labquilrs}
  A {\em labelled, quilted Riemann surface} is a quilted Riemann
  surface $S$ as in the previous paragraph together with labels: we
  label each patch $S_i$ with a target manifold $Z_i$, each seam
  $\sigma_{ij}$ with a Lagrangian $L_{ij}\subset Z_i^-\times Z_j$ and
  each component $b$ of $\partial S$ with a Lagrangian $L_b\subset
  Z_{i(b)}$. A {\em quilted map} modelled on a labelled, quilted
  Riemann surface is a collection of maps $u_i\colon S_i\to Z_i$
  satisfying the boundary and seam conditions
  \[u_{i(b)}(z)\in L_{i(b)}\mbox{ for }z\in b,\qquad (u_i(z),u_j(z))\in L_{ij}\mbox{ for }z\in\sigma_{ij}.\]
\end{dfn}

One can define higher $A_\infty$-operations for collections of
generalised Lagrangian submanifolds, in a similar way to the
differential, by counting pseudoholomorphic quilted Riemann surfaces
modelled on punctured discs with strip-like ends asymptotic to
generalised intersection points. The underlying moduli spaces of
quilted Riemann surfaces which one uses have natural compactifications
by allowing nodal degenerations. The $A_\infty$-equations follow from
the fact that these compactifications are homeomorphic to associahedra
\cite{M,MWW,MW}.

\begin{rmk}
In fact, with this notion of pseudoholomorphic quilt, it is not
possible to achieve transversality for moduli spaces when some but not
all of the maps $u_i$ are constant. We must allow {\em folded
  perturbations} of the Floer equation. To define these perturbations,
we pick a collection of squares
$(-\delta,\delta)\times(p_i-\epsilon,p_i+\epsilon)$ in $\RR\times
[0,1]$ centred on the seams. For each square, let $\tau_i$ be the
reflection in the seam. Consider the map
$(-\delta,\delta)\times[p_i,p_i+\epsilon]\to C_i^-\times C_{i+1}$
defined by $z\mapsto (u_i(\tau(z)),u_{i+1}(z))$. We require that this
solves Floer's equation in $C_i^-\times C_{i+1}$ for a $z$-dependent
almost complex structure $J$ which has the split form $(-J_i)\oplus
J_{i+1}$ except in a half-ball in
$(-\delta,\delta)\times[p_i,p_i+\epsilon]$ centred at $(0,p_i)$. In
more general quilted domains, discussed above, the seams are required
to be real analytic; this allows one to fold locally along seams and,
hence, make sense of folded perturbations on balls centred on
seams. See also \cite{WWconstant}.
\end{rmk}

In order that the pseudoholomorphic quilt moduli spaces above define
an $A_\infty$-structure, we will require some monotonicity conditions
to hold. First, there is the requirement that each generalised
Lagrangian submanifold is monotone:

\begin{dfn}
Let $Z$ be a compact, monotone, symplectic manifold. A generalized
Lagrangian submanifold $\underline{L}$ of $Z$ is monotone if all the
$A_i$ are monotone with the same monotonicity constant and if all the
$L_{i,i+1}$ are monotone Lagrangian branes with minimal Maslov number
$N_{L_{i,i+1}}\geq 2$. Note that if $Z_{i(b)}$ is exact and
$c_1(Z_{i(b)})=0$ and $L_b$ is an exact, Maslov zero Lagrangian, then
it is $\tau$-monotone for any $\tau$, since both area and index are
zero.
\end{dfn}

To ensure that the Floer differential squares to zero, we also need
control on Maslov 2 disc bubbling:

\begin{dfn}
If $\underline{M}=\{M_{i,i+1}\subset C_i^-\times
C_{i+1}\}_{i=0}^{m-1}$ is a monotone Lagrangian correspondence from
$\{pt\}$ to $\{pt\}$ then define
\[\mathfrak{m}_0(\underline{M})=\sum_{i=0}^{m-1}\mathfrak{m}_0(M_{i,i+1}),\]
where $\mathfrak{m}_0(M_{i,i+1})$ is the count of Maslov 2 discs in
$C_i^-\times C_{i+1}$ with boundary on $M_{i,i+1}$ and one point
constraint on the boundary.
\end{dfn}

\begin{rmk}
If $\underline{K}$ and $\underline{L}$ are monotone generalised
Lagrangian submanifolds of $Z$, and $\partial$ denotes the
differential on $\CF(\underline{K},\underline{L};\kk)$ then
	$\partial^2=\mathfrak{m}_0((\underline{K},\underline{L}^T))\OP{Id}$ (proved as in \cite{OhI,OhII}).
\end{rmk}

Finally, to ensure that counts of pseudoholomorphic quilts defining
the $A_\infty$-structure are finite, and to control bubbling, one
requires a {\em simultaneous} monotonicity condition to hold
{\cite[Definition 3.6]{WWpseudoquilt}}. We summarise these conditions
in the following definition:

\begin{dfn} \label{admissible} 
A (possibly infinite) set of generalized Lagrangian submanifolds $\{
\underline{K}_i\}_{i \in I }$ of $Z$
is called \emph{an admissible set} if there exists a $\tau\geq 0$ such that:
\begin{itemize}
\item each $\underline{K}_i$ is a $\tau$-monotone generalised
  Lagrangian submanifold;
\item for all $i,j\in I$, we have
  $\mathfrak{m}_0(\underline{K}_i,\underline{K}_j^T)=0$; and
\item for all continuous quilted maps modelled on a compact, genus
  zero quilted Riemann surface, where each seam is labelled by a
  component of one of the $\underline{K}_i$, we have
\end{itemize}
\begin{equation}
\label{eq:globmon}\OP{area}(\underline{u})=\tau\OP{ind}(\underline{u}).
\end{equation}
\end{dfn}

Equation \eqref{eq:globmon} is automatically satisfied if the
generalised Lagrangian submanifolds are monotone with the same
monotonicity constant and all the manifolds $Z_i$ are simply-connected
{\cite[Remark 3.7]{WWpseudoquilt}}.

In \cite{MWW}, Mau, Wehrheim and Woodward defined an
$A_\infty$-category $\fF^{\#}(Z;\kk)$ which they call the {\em
  extended Fukaya category}, whose objects are given by an admissible
set of generalised Lagrangian submanifolds of $Z$.

Given a monotone Lagrangian correspondence $L\subset Y^{-}\times Z$
with a brane structure, {\cite[Theorem 1.1]{MWW}} tells us that there
is an $A_\infty$-functor
\begin{gather*}
\Phi(L)\colon\fF^{\#}(Y;\kk)\to\fF^{\#}(Z;\kk)\\
\Phi(L)(L_{01},\ldots,L_{k-1,k})=(L_{01},\ldots,L_{k-1,k},L)
\end{gather*}
if one can ensure that this assignment preserves the requirements of
admissibility as given in Definition \ref{admissible}. Moreover, the
assignment $\Phi\colon\fF(Y^{-}\times Z;\kk)\to\nufun(\fF^{\#}(Y;\kk),\fF^{\#}(Z;\kk))$ is an
$A_\infty$-functor.

Below, whenever we describe a general result about the functors
$\Phi(L)$, we implicitly assume that the functor preserves
admissibility (a condition that we check when we make an explicit use
of these functors).

\begin{dfn}
Let $L\subset Z_1^{-}\times Z_2$ and $M\subset Z_2^{-}\times Z_3$ be
Lagrangian correspondences. We say the pair $(L,M)$ is composable if:
\begin{itemize}
\item The intersection
\[L\times_{Z_2}M:=(L\times M)\cap \left(Z_1^{-}\times\Delta_{Z_2}^{-}\times Z_3\right)\]
is transverse.
\item The projection $\OP{pr}_{Z_1^{-}\times Z_3}\colon Z_1^{-}\times
  Z_2\times Z_2^{-}\times Z_3\to Z_1^{-}\times Z_3$ restricted to
  $L\times_{Z_2}M$ is an embedding.
\end{itemize}
we define $L\circ M=\OP{pr}_{Z_1^{-}\times Z_3}(L\times_{Z_2}M)$ and
we call this the {\em geometric composition} of $L$ and $M$.
\end{dfn}

\begin{thm}[{\cite[Theorem 1.2]{MWW}}]
If $Z_1,Z_2,Z_3$ are monotone symplectic manifolds with the same
monotonicity constants and $L\subset Z^{-}_1\times Z_2$ and $M\subset
Z_2^{-}\times Z_3$ are composable monotone generalised Lagrangian
correspondences with brane structures then
\[\Phi(L)\circ\Phi(M)\simeq\Phi(L\circ M)\]
where $\simeq$ denotes quasi-isomorphism of $A_\infty$-functors.
\end{thm}
Mostly, we will use this result only in the case $Z_1=\{pt\}$, where
it takes the slightly simpler form
\[(L,M)\simeq L\circ M,\]
where $\simeq$ denotes a quasi-isomorphism in $\fF^{\#}(Z_3;\kk)$. As
explained in {\cite[Proposition 7.2.5]{MWW}}, this simpler result is a
reinterpretation of the cohomology-level geometric composition theorem
\cite{LekiliLipyanskiy,WWcomposition}: given a test object $N$, the
quasi-iso\-mor\-ph\-ism (Y-map)
\[\CF(L\circ M,N) \to \CF((L,M),N)\]
constructed in \cite{LekiliLipyanskiy} is chain homotopic, by a
deformation of the quilt domain, to a map of the form $\mu^2(\eta,-)$
for a Floer cocycle $\eta\in \CF((L,M),L\circ M)$. Similarly, the
inverse Y-map has the form $\mu^2(-,\eta')$ for a cocycle
$\eta'\in\CF(L\circ M,(L,M))$. The cocycles $\eta$ and $\eta'$ provide
the desired quasi-isomorphism.

\begin{rmk}\label{rmk:yend-mon}
We remark that to define the cocycle $\eta$ using the technology of
\cite{LekiliLipyanskiy}, one must first establish monotonicity in the
sense of Equation \eqref{eq:globmon} for annuli in $Z_2^-\times Z_3$
with boundary on $L\times (L\circ M)$ and $M$
{\cite[Corrigendum]{LekiliLipyanskiy}}. For example, this is automatic
if the image of $\pi_1(L\times (L\circ M))$ in $\pi_1(Z_2^-\times
Z_3)$ is torsion. Later (Section \ref{lma:monotone}) we will see an
example where monotonicity holds for different reasons.
\end{rmk}

Only in Corollary \ref{cor:nolagorb} will we use the geometric
composition theorem in the more general situation.

\begin{exm}
If $\Delta\subset X^{-}\times X$ denotes the diagonal Lagrangian then
the associated functor $\Phi(\Delta)$ is quasi-isomorphic to the
identity functor. Moreover, recall that $\Phi$ is itself functorial,
so we get a map on morphisms
\[\Phi^1\colon\hom(\Delta,\Delta)\to\hom(\OP{id},\OP{id}).\]
The self-morphisms of the identity functor are the pre-natural
transformations $(T^0,T^1,\ldots)$, where $T^0$ in particular assigns
to each element $L$ a morphism $T^0_L\in\hom(L,L)$. The composition of
$\Phi^1$ with this gives a map $\hom(\Delta,\Delta)\to\hom(L,L)$ for
each object $L$. On cohomology this map is the closed-open map
$CO^0\colon \QH(X;\kk)\to \HF(L,L;\kk)$ {\cite[Remark
    7.3]{WWfunctoriality}}.
\end{exm}

\begin{exm}\label{exm:projfunc}
If $L_1$ and $L_2$ are monotone Lagrangian branes then the functor
$\Phi(L_1\times L_2)$ is quasi-isomorphic to a {\em projection
  functor} $\mathcal{I}_{L_1,L_2}$ which acts in the following way on
other monotone Lagrangian branes:
\[\mathcal{I}_{L_1,L_2}(L)=\CF(L_1,L)\otimes L_2.\]
See {\cite[Lem. 7.4]{AbouzaidSmith}} for a proof of this fact and
{\cite[Section 3c]{Seidel}} for the definition of the tensor product
of an object and a cochain complex.
\end{exm}

\begin{rmk}\label{rmk:projfunc}
In Example \ref{exm:projfunc}, $\CF(L_1,L)\otimes L_2$ is the tensor
product of an object $L_2$ and a cochain complex $\CF(L_1,L)$. We note
that this contains $L_2$ as a direct summand provided $\HF(L_1,L)\neq
0$. To see this, suppose that $\eta\in\CF(L_1,L)$ is a closed element
not contained in the image of the differential. Let $S$ be a subspace
of $\CF(L_1,L)$ which (a) contains the image of the differential, and
(b) is complementary to the span of $\eta$. Since this contains the
image of the differential, it is a subcomplex. The splitting
$\CF(L_1,L)=S\oplus\langle\eta\rangle$ allows us to define an
idempotent chain map $\CF(L_1,L)\to\CF(L_1,L)$ whose image is
$\langle\eta\rangle$.  This gives an idempotent summand of
$\CF(L_1,L)\otimes L_2$ quasi-isomorphic to $L_2$.
\end{rmk}

One important consequence of the theory is:

\begin{cor}
  To prove split-generation of $D^{\pi}\fF(X;\kk)$ by a collection of
  Lagrangians $L_i$ it suffices to prove that $\Delta$ is
  split-generated in $D^{\pi}\fF(X^{-}\times X;\kk)$ by product
  Lagrangians $L_i\times L_j$.
\end{cor}
\begin{proof}
  If this holds then, upon applying the functor $\Phi$, we see that
  the identity functor is split-generated by the functors
  $\Phi(L_i\times L_j)$. Therefore, by Lemma \ref{lma:twcpxfunctors},
  $P$ is quasi-isomorphic to an object split-generated by the objects
  $\Phi(L_i\times L_j)(P)=\CF(L_i,P)\otimes L_j$, which is clearly in
  the subcategory split-generated by $L_j$.
\end{proof}

\subsection{Flat line bundles}\label{sct:locsysquilt}

Let $\kk$ be a field. A {\em flat $\kk$-line bundle} on a Lagrangian
$L$ is a rank 1 $\kk$-vector bundle over $L$ equipped with parallel
transport maps along paths which are independent of the path up to
homotopy relative to its endpoints. We will drop the $\kk$ from the
notation in what follows.

The definition of the Fukaya category $\fF(Z;\kk)$ can be extended so
that its objects are pairs $(L,E)$ of Lagrangians with flat line
bundles\footnote{One could also use higher rank local systems, as in
  \cite{AbouzaidNearby}, but this can cause issues in the presence of
  Maslov 2 discs, as exploited in the work of Damian
  \cite{Damian}; see \cite{Konstantinov} for a fuller discussion of
  these issues.}. We explain how to define an extended Fukaya
category of generalised Lagrangian submanifolds equipped with flat
line bundles.

\begin{dfn}
  Let $\underline{L}=\{L_{i,i+1}\}_{i=0}^{k-1}$ be a generalised
  Lagrangian submanifold of $Z$ and let $\kk$ be a field. A flat line
  bundle on $\underline{L}$ is a $k$-tuple $\{E_{i,i+1}\}_{i=0}^{k-1}$
  of flat line bundles $E_{i,i+1}$ on $L_{i,i+1}$. Given a flat line
  bundle $\underline{E}$ on $\underline{L}$, we define
  $\underline{E}^T$ to be the flat line bundle on $\underline{L}^T$
  given by $E^\vee_{i,i+1}$ on $L^-_{i,i+1}$ for $i=k-1,\ldots,1$.
\end{dfn}

If $\underline{E}$ is a flat line bundle on $\underline{K}$ then we
get a flat line bundle
\[E_{01}\boxtimes\cdots\boxtimes E_{k-1,k}:=\left(p_{01}^*E_{01}\right)\otimes\cdots\otimes\left(p_{k-1,k}^*E_{k-1,k}\right)\]
on $K_{01}\times\cdots\times K_{k-1,k}$, where $p_{i,i+1}$ denotes the
projection to the factor $K_{i,i+1}$. Given flat line bundles
$\underline{E}$ and $\underline{F}$ on $\underline{K}$ and
$\underline{L}$ respectively, we get a bundle
\[\underline{E}\boxtimes\underline{F}^\vee:=E_{01}\boxtimes\cdots\boxtimes E_{k-1,k}\boxtimes F^\vee_{\ell-1,\ell}\boxtimes\cdots\boxtimes F^\vee_{0,1}\]
on $K_{01}\times\cdots\times K_{k-1,k}\times
L^-_{\ell-1,\ell}\times\cdots\times L^-_{01}$. Given a generalised
intersection point $\underline{x}\in\underline{K}\cap\underline{L}$,
the point
$(x_0,x_1,x_1,x_2,\ldots,x_{k+\ell-1},x_{k+\ell-1},x_{k+\ell})$ lives
in $K_{01}\times\cdots\times K_{k-1,k}\times
L^-_{\ell-1,\ell}\times\cdots\times L^-_{01}$ and we can take the
fibre of $\underline{E}\boxtimes\underline{F}^T$ at this point. We
write $\underline{E}\boxtimes\underline{F}^T(\underline{x})$ for this
vector space.

\begin{dfn}
Given $(\underline{K},\underline{E})$ and
$(\underline{L},\underline{F})$, generalised Lagrangian
correspondences equipped with flat line bundles, define
\[\hom((\underline{K},\underline{E}),(\underline{L},\underline{F}))=\bigoplus_{\underline{x}\in\underline{K}\cap\underline{L}}\underline{E}\otimes\underline{F}^T(\underline{x}).\]
\end{dfn}

We can define a differential on this morphism space by counting
pseudoholomorphic quilted strips as before. Each quilted strip
$\underline{u}=(u_1,\ldots,u_{k+\ell-1})$, asymptotic at $-\infty$ to
$\underline{x}$ and at $+\infty$ to $\underline{y}$, defines a path
\[(u_1(0,t),u_1(p_1,t),u_2(p_1,t),u_2(p_2,t),\ldots,u_{k+\ell-1}(p_{k+\ell-1},t),u_{k+\ell}(1,t))\]
in $K_{01}\times\cdots\times K_{k-1,k}\times
L^-_{\ell-1,\ell}\times\cdots\times L^-_{01}$ and, hence, a monodromy
map
\[T_{\underline{u}}\colon\underline{E}\otimes\underline{F}^T(\underline{x})\to\underline{E}\otimes\underline{F}^T(\underline{y}).\]
The contribution of the quilted strip $\underline{u}$ to the
differential is precisely this linear map.

Higher $A_\infty$-products are handled similarly. If a seam $\sigma$
connects an incoming and an outgoing end then the parallel transport
along $\sigma$ is used to define the contribution of the quilt in the
same way as in the definition of the differential. There is an
additional complication when a seam connects two incoming strip-like
ends (not necessarily distinct) or an outgoing end to itself.

In the case of the incoming ends, this is resolved as follows. Suppose
that we have a quilt with several incoming strip-like ends, two of
which ($p$ and $q$) are labelled with Floer cochain groups
$\hom((\underline{K},\underline{E}),(\underline{L},\underline{F}))$
and
$\hom((\underline{K}',\underline{E}'),(\underline{L}',\underline{F}'))$. Suppose
moreover that there is a seam whose ends connect these punctures. This
seam will be labelled with $(L,E)$ at the $p$-end and $(L^-,E^\vee)$
at the $q$-end. Suppose that $x$ is the asymptote at the $p$-end and
$y$ is the asymptote at the $q$-end. The contribution of this quilt
$A_\infty$-operation will be a linear map of the form \[Q\colon
H_1\otimes\hom((\underline{K},\underline{E}),(\underline{L},\underline{F})\otimes
H_2\otimes\hom((\underline{K}',\underline{E}'),(\underline{L}',\underline{F}'))\otimes
H_3\to H_4\] where the vector spaces $H_i$ are tensor products of
morphism spaces. The tensor product on the left-hand side includes
$E(x)\otimes E^\vee(y)$. Parallel transport along the seam allows us
to identify this with $E(x)\otimes E^\vee(x)$, which has a canonical
contraction $E(x)\boxtimes E^\vee(x)\to\kk$. We require that the map
$Q$ factors through this contraction.

In the case of a seam connecting an outgoing end to itself, we make a
similar argument using the canonical map $\kk\to E\otimes E^\vee$
which sends $1$ to the identity. This is sufficient to determine all
quilt contributions.

\begin{dfn}
Suppose that $(K\subset Y,L\subset Y^{-}\times Z)$ is a composable
generalised Lagrangian correspondence. Let $K$ and $L$ be equipped
with flat line bundles, $E$ and $F$. Note that there is a natural map
\[p\colon K\times_YL\to K\times L\]
as $K\times_YL$ includes into $L$ and admits a projection to $K$. Let
\[E\boxtimes F:=\left(\OP{pr}_K^*E\right)\otimes\left(\OP{pr}_L^*F\right).\]
We can define a flat line bundle $E\circ F$ on $K\circ L$ by
$p^*\left(E\boxtimes F\right)$.
\end{dfn}

As before, it is easy to associate to any pair $(L,F)$ comprising a
Lagrangian correspondence with a flat line bundle, an
$A_\infty$-functor $\Phi(L,F)$ between these enlarged categories. The
proof of the geometric composition theorem generalises readily to
show:
\[\Phi(L,F)(K,E)=(K\circ L,E\circ F).\]

\subsection{Non-compact Lagrangians}\label{sct:wrappedquilt}

Let $(W,\theta)$ be a Liouville manifold with $2c_1(W)=0$. Suppose
there is a compact Liouville subdomain $W^{in}$ whose complement is
isomorphic to
\[\left([1,\infty)\times \partial W^{in},r\left(\theta|_{\partial W^{in}}\right)\right)\]
where $r$ is the coordinate on $[1,\infty)$. The wrapped Fukaya
category of $W$ is an $A_\infty$-category whose objects are exact,
Maslov zero, Lagrangian submanifolds, equipped with orientations,
grading, relative spin structures and functions $f_L$ such that
$\theta|_L=df_L$. Additionally, and importantly, we require that
outside a compact set, $\theta|_L=0$. This constrains the behaviour
of $L$ at infinity: it must have the form $[R,\infty)\times \Lambda$
for some Legendrian submanifold $\Lambda\subset\partial W^{in}$.

We will work with the model for the wrapped Fukaya category explained
in \cite{AbouzaidGenerationCriterion} (an alternative model appears in \cite{AbouzaidSeidel}, which is known to agree with the model in \cite{AbouzaidGenerationCriterion}). Consider the class of
Hamiltonian functions $\mathcal{H}(W)$ which agree with $r^2$ outside
a compact set. Given two Lagrangians $L_1,L_2$ and a Hamiltonian
$H\in\mathcal{H}(W)$, an $X_H$-chord is an integral curve $x$ of the
Hamiltonian vector field $X_H$ with $x(0)\in L_1$ and $x(1)\in
L_2$. Define $\CW(L_1,L_2,H;\kk)$ to be the $\kk$-vector space
generated by the chords. The $A_\infty$-operations are defined using
moduli spaces of solutions to Floer's equation in a standard
way. There is a subtlety: for example, the natural composition is
$\CW(L_2,L_3,H;\kk)\otimes \CW(L_1,L_2,H;\kk)\to \CW(L_1,L_3,2H;\kk)$;
to define an $A_\infty$-structure on $\CW(\cdot,\cdot,H;\kk)$ itself
requires a careful choice of Hamiltonian perturbations depending on
the domain and the use of the Liouville flow to ``rescale'' from $kH$
to $H$ (see \cite{AbouzaidGenerationCriterion}).

We now fix:
\begin{itemize}
\item a Liouville manifold $(W,\theta)$ with $2c_1(W)=0$;
\item a simply-connected, compact, monotone symplectic manifold $Z$
  with monotonicity constant $\tau$; and
\item a compact, $\tau$-monotone Lagrangian brane $C\subset W^-\times Z$.
\end{itemize}
We define $\fF_C(Z;\kk)$, an extended category of generalised
Lagrangians in $Z$ having a very special form:

\begin{dfn}[Objects of $\fF_C(Z;\kk)$]\label{dfn:objects}
The objects of $\fF_C(Z;\kk)$ come in two flavours:
\begin{enumerate}
\item[(A)] compact $\tau$-monotone Lagrangian branes $M\subset Z$.
\item[(B)] generalised Lagrangian correspondences $(L,C)$ where
  $L\subset W$ is an exact, Maslov zero Lagrangian brane in $\wW(W)$.
\end{enumerate}
We also require the objects to form an admissible set (Definition
\ref{admissible}).
\end{dfn}

\begin{rmk}\label{rmk:admissibility}
  Checking admissibility when there is non-trivial fundamental group
  can be difficult. However, in the situation of Definition
  \ref{dfn:objects}, there is a sequence of modifications to a general
  genus zero quilted map which allow us to reduce to checking
  monotonicity for annuli with boundary on $L\times M$ and $C$ where
  $(L,C)$ is an object of type (B) and $M$ is some fixed object of
  type (A) (it does not matter which). These modifications have the
  form described in the following lemma, where we first homotope the
  quilted map to be patchwise-constant in some region and then perform
  an excision to change the map in that region. The lemma guarantees
  that, at each stage in the sequence, the modifications preserve
  monotonicity of the quilted map:
\end{rmk}

\begin{lma}[Excision lemma]\label{lma:excision}
  Let $T_{12}$ and $T_{34}$ be two labelled, quilted Riemann surfaces
  (see Definition \ref{dfn:labquilrs}). Pick two simple closed curves
  $\gamma\subset T_{12}$ and $\gamma'\subset T_{24}$ in their
  interiors intersecting the seams transversally. Let us write
  $T_{12}=T_1\cup_\gamma T_2$ and $T_{34}=T_3\cup_{\gamma'}
  T_4$. Suppose that the quilt and label data along $\gamma$ and
  $\gamma'$ match, so that $T_{13}=T_1\cup T_3$ and $T_{24}=T_2\cup
  T_4$ define new quilted domains. Suppose also that, for each
  $i\in\{1,2,3,4\}$, there is a quilted map $v_i$ modelled on $T_i$
  such that $v_{ij}:=v_i\cup v_j$ defines a continuous quilted map on
  $T_{ij}$ for $ij\in\{12,34,13,24\}$. Then $v_{12}$ satisfies
  Equation \eqref{eq:globmon} if $v_{34},v_{13},v_{24}$ all satisfy
  Equation \eqref{eq:globmon}.
\end{lma}
\begin{proof}
  This is an immediate consequence of additivity for area and the
  excision theorem for index. See {\cite[Section
      2.4]{LekiliLipyanskiy}} for an explanation of excision in the
  quilted context.
\end{proof}

\begin{figure}
\begin{center}
\labellist
\pinlabel $\color{black}{T_2}$ at 110 200
\pinlabel $\color{black}{T_1}$ at 110 40
\pinlabel $\color{black}{T_4}$ at 440 200
\pinlabel $\color{black}{T_3}$ at 440 40
\pinlabel $\color{black}{T_{12}}$ at 5 220
\pinlabel $\color{black}{T_{34}}$ at 545 220
\pinlabel $\color{black}{\gamma}$ at 152 73
\pinlabel $\color{black}{\gamma'}$ at 390 76
\endlabellist
\includegraphics[width=200px]{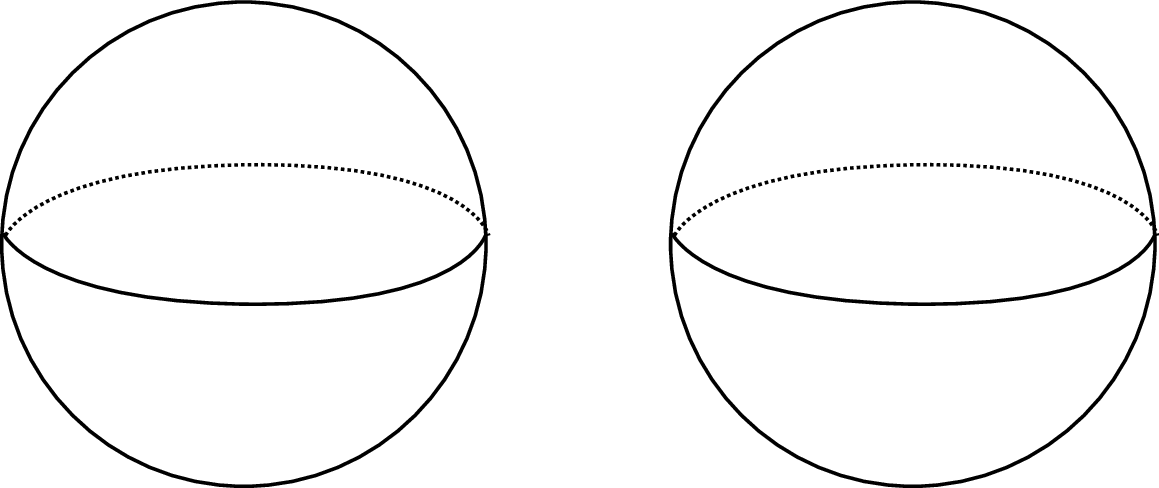}\vspace{0.3cm}
\labellist
\pinlabel $\color{black}{T_3}$ at 110 200
\pinlabel $\color{black}{T_1}$ at 110 40
\pinlabel $\color{black}{T_2}$ at 440 200
\pinlabel $\color{black}{T_4}$ at 440 40
\pinlabel $\color{black}{T_{13}}$ at 5 10
\pinlabel $\color{black}{T_{24}}$ at 545 10
\endlabellist
\includegraphics[width=200px]{excision}
\end{center}
\caption{Excision.}
\label{fig:excision}
\end{figure}

\begin{dfn}[Morphism spaces of $\fF_C(Z;\kk)$]\label{dfn:morphisms}
The morphism spaces come in four flavours. Let $M_i\subset Z$, $i=1,2$ be
objects of type (A) and $(L_i,C)$ be objects of type (B). Define
\begin{align*}
\hom(M_1,M_2)&=\CF(M_1,M_2;\kk)\mbox{ in }Z\\
\hom(M_1,(L_2,C))&=\CW(L_2\times M_1,C;\kk)\mbox{ in }W^{-}\times Z\\
    \hom((L_1,C),M_2)&=\CW(M_2\times L_1,C^{-};\kk)\mbox{ in }Z^{-} \times W\\
    \hom((L_1,C),(L_2,C))&=\CW(L_1\times\Delta_Z\times L_2,C\times C^{-};\kk)\mbox{ in }W\times Z^{-}\times Z\times W^{-}.
\end{align*}
Here, $\CW$ indicates wrapping with respect to a Hamiltonian
$\underline{H}$ defined on the target manifold $A_1\times \cdots\times
A_k$ which has the form $\sum_{i=1}^k H_i$ where $H_i\colon A_i\to\RR$
is in $\mathcal{H}(W)$ whenever $A_i=W$ or $W^{-}$. We will assume
that our Lagrangians have been perturbed and our Hamiltonians chosen
so that all Hamiltonian chords are non-degenerate.
\end{dfn}

\begin{rmk}
Since our correspondence is compact, we never have to take Floer
cohomology between two non-compact Lagrangians (only between a compact
and a non-compact Lagrangian). For this reason, although our Floer
complexes are wrapped, they are finite-dimensional.
\end{rmk}

The $A_\infty$-structure can be defined for $\mathcal{F}_C(Z;\kk)$ in
the same way that Mau, Wehrheim and Woodward \cite{MWW} define the extended
Fukaya category in their setting. For example, Figure \ref{fig:quilt2}
illustrates a quilt contributing to a $\mu^3$-product in the category
(in fact, the $\mu^3$-product is defined by counting pseudoholomorphic
quilts modelled on this picture with a suitably chosen family of
perturbation data, as explained in \cite{MWW}).

\begin{figure}
\begin{center}
\labellist
\pinlabel $\color{black}{Z}$ at 150 120
\pinlabel $\color{black}{W}$ at 175 152
\pinlabel $\color{black}{W}$ at 154 80
\pinlabel $\color{black}{W}$ at 104 76
\pinlabel $\color{black}{W}$ at 110 153
\pinlabel $\color{black}{C}$ at 80 189
\pinlabel $\color{black}{C}$ at 70 34
\pinlabel $\color{black}{L}$ at 180 180
\pinlabel $\color{black}{L}$ at 184 76
\pinlabel $\color{black}{L}$ at 70 73
\pinlabel $\color{black}{L}$ at 70 153
\pinlabel $\color{black}{C}$ at 270 135
\pinlabel $\color{black}{C}$ at 280 115
\endlabellist
\includegraphics[width=240px]{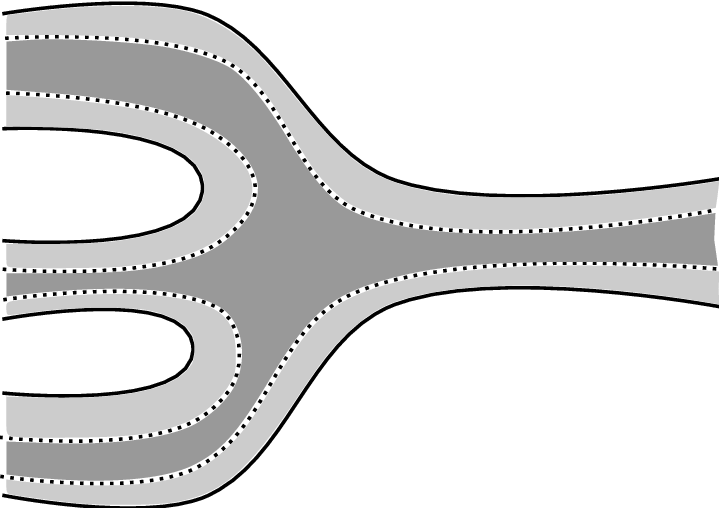}
\end{center}
\caption{A quilt contributing to the $A_\infty$-operation $\mu^3$ on
  $\hom_{\fF_C(Z;\kk)}((L,C),(L,C))$. The light patches map to $W$;
  the dark patch maps to $Z$; the edges map to the Lagrangian
  $L\subset W$; the dotted lines are seams which map to the Lagrangian
  correspondence $C\subset W^{-}\times Z$.}
\label{fig:quilt2}
\end{figure}

There are two points where one must be careful with this construction:
in proving transversality and in proving compactness. We deal with
these considerations in the next two subsections.

\subsubsection{Transversality}

As before, we must make domain-dependent folded perturbations in balls
in the domain centred on the seams. We will need to use non-split
almost complex structures on $Z^-\times W$ or $W^-\times Z$, but we
will require that the almost complex structure is split outside the
compact region $Z^-\times W^{in}$ or $(W^{in})^-\times Z$. This is
enough to prove transversality because the seam condition is given by
a compact Lagrangian correspondence contained in $Z^-\times W^{in}$ or
$(W^{in})^-\times Z$, so the pseudoholomorphic curve must pass through
this region.

\subsubsection{Compactness}

To prove compactness, we need to show that patches of
pseudoholomorphic quilts which map to $W$ with given asymptotic
behaviour at the punctures stay in a given compact region of
$W$. Suppose $S_i$ is a patch in a quilted Riemann surface and $u$ is
a pseudoholomorphic quilt for which $u_i\colon S_i\to W$ lands in
$W$. The part of $u_i$ which escapes from $W^{in}$ is a genuine
solution to Floer's equation, as we require our domain-dependent
almost complex structures to be split outside of regions which project
to $W^{in}$ (see the remark on transversality). For fixed asymptotics,
solutions to Floer's equation are forced to remain in a compact region
of $W$, and for fixed input asymptotics, there can only be finitely
many output asymptotics; this is explained in detail for $\mu^1$ in
{\cite[Lem. 2.4 and 2.5]{AbouzaidGenerationCriterion}}.

\subsubsection{$A_\infty$-functor}

Exactly as in \cite{MWW}, the correspondence $C$ can be used to define
an $A_\infty$-functor $\Phi(C)\colon\bB\to\fF_C(Z;\kk)$, defined on
the subcategory $\bB$ of $\wW(L)$ consisting of objects $L$ such that
$(L,C)$ is a brane of type (B) in $\fF_C(Z;\kk)$. This functor acts on
objects by $\Phi(C)(L)=(L,C)$. The component maps in the
$A_\infty$-functor are defined by counts of pseudoholomorphic quilts,
see Figure \ref{fig:quilt1}. Transversality, compactness and
monotonicity issues can be dealt with in the same way as in setting up
the $A_\infty$-category $\fF_C(Z;\kk)$.

If $(L,C)$ is composable and $L\circ C$ is monotone then geometric
composition works as explained in Section \ref{sct:stdquilt}; the
monotonicity condition explained in Remark \ref{rmk:yend-mon} is built
into the assumption of admissibility in Definition \ref{dfn:objects}.

\begin{figure}
\begin{center}
\labellist
\pinlabel $\color{black}{W}$ at 210 130
\pinlabel $\color{black}{Z}$ at 310 130
\pinlabel $\color{black}{L_3}$ at 180 180
\pinlabel $\color{black}{L_0}$ at 184 76
\pinlabel $\color{black}{L_1}$ at 70 73
\pinlabel $\color{black}{L_2}$ at 70 153
\pinlabel $\color{black}{C}$ at 270 130
\endlabellist
\includegraphics[width=240px]{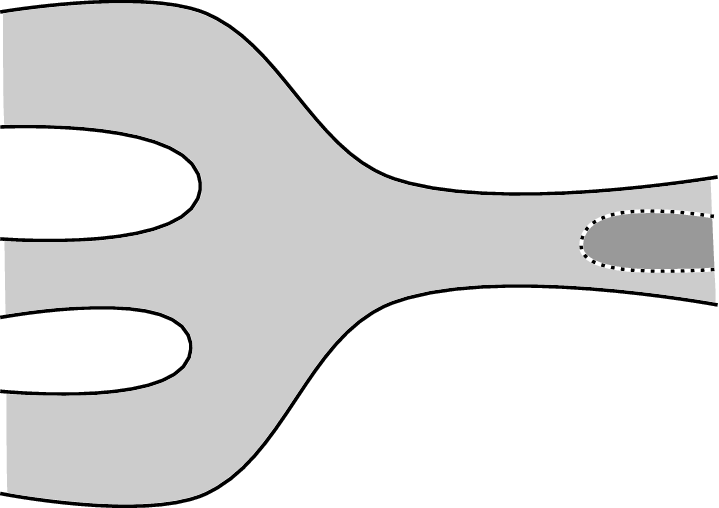}
\end{center}
\caption{A quilt contributing to the $A_\infty$-functor
  $\Phi(C)\colon\wW(W)\to\fF_C(Z;\kk)$, specifically to
  $\Phi(C)^3\colon\hom(L_2,L_3)\otimes\hom(L_1,L_2)\otimes\hom(L_0,L_1)\to\hom_{\widetilde{\fF}(Z)}((L_0,C),(L_3,C))$.}
\label{fig:quilt1}
\end{figure}

We summarise this discussion in the following statement:

\begin{thm}
Let $\bB$ denote the subcategory of $\wW(W)$ consisting of exact
Lagrangian branes $L$ such that $(L,C)$ is a brane of type (B) in
$\fF_C(Z;\kk)$. There is an $A_\infty$-functor
\[\Phi(C)\colon\bB\to\fF_C(Z;\kk)\]
which acts on objects as
\[\Phi(C)(L)=(L,C).\]
Moreover, if $(L,C)$ is a composable correspondence in $\fF_C(Z;\kk)$
then
\[(L,C)\simeq L\circ C.\]
\end{thm}

\section{Hamiltonian group actions}\label{sct:ham}

\subsection{Moment correspondence functor}\label{sct:mcf}

Let $G$ be a compact Lie group, let $(X,\omega)$ be a simply-connected
monotone symplectic manifold, and suppose that $X$ admits a
Hamiltonian action of $G$ with equivariant moment map $\mu\colon
X\to\mathfrak{g}^*$.  The {\em moment correspondence} is the
Lagrangian submanifold
\[C:=\{(g,v,x,y)\in (T^*G)^{-}\times X^{-}\times X\ :\ v=\mu(gx),\ y=gx\}.\]
Here, we are using left-multiplication to trivialise the cotangent
bundle $T^*G\cong G\times\mathfrak{g}^*$ and write $(g,v)$ for a point
in $G\times\mathfrak{g}^*$.

\begin{lma}
The moment correspondence is a monotone Lagrangian submanifold.
\end{lma}
\begin{proof}
The relative homotopy exact sequence is
\begin{gather*}\cdots\to\pi_2(C)\to\pi_2(T^*G\times X\times X)\to\pi_2(T^*G\times X\times X,C)\to\\
\to\pi_1(C)\to\pi_1(T^*G\times X\times X)\to\cdots\end{gather*} We
know that $\pi_2(G)=0$ and that the map
$\pi_1(G)=\pi_1(C)\to\pi_1(T^*G\times X\times X)$ is induced by the
map $(g,x)\mapsto(g,\mu(gx),x,gx)$ and is therefore injective. This
implies $\pi_2(T^*G\times X\times X,C)$ is a quotient of
$\pi_2(X\times X)$ and monotonicity follows from monotonicity of $X$.
\end{proof}

As the moment correspondence is compact and monotone, the quilt
machinery from Section \ref{sct:wrappedquilt} defines for us an
$A_\infty$-functor
\[\Phi(C)\colon\bB\to\fF_C(X^{-}\times X;\kk)\]
where $\bB$ is the subcategory of $\wW(T^*G;\kk)$ consisting of exact
Lagrangian branes $L$ such that $(L,C)$ is a brane of type (B) in
Definition \ref{dfn:objects}. The following observation is crucial to
our paper.

\begin{lma}\label{lma:diagcomp}
  Let $T^*_1G$ denote the cotangent fibre at the identity element in
  $G$. The generalised correspondence $(T^*_1G,C)$ is composable and
  its geometric composition is the diagonal Lagrangian $\Delta\subset
  X^-\times X$.
\end{lma}
\begin{proof}
  We have a transverse intersection $(T^*_1G\times X^-\times X)\cap
  C=\{(1,\mu(x),x,x)\}$ and the projection of this to $X^-\times X$ is
  the embedding of the diagonal.
\end{proof}

Since $T_1^*G$ is simply-connected, the object $(T^*_1G,C)$ satisfies
monotonicity for annuli with boundaries on $T^*_1G\times\Delta$ and
$C$ and hence, by Remark \ref{rmk:admissibility}, can be incorporated
as an object of type (B) into $\fF_C(X^-\times X;\kk)$. This implies:

\begin{dfn}[Moment correspondence functor]
  If we restrict $\Phi(C)$ to the subcategory $\langle T^*_1G\rangle$
  generated by $T^*_1G$, then, after a quasi-isomorphism, we can
  assume that it lands in the usual Fukaya category $\fF(X^-\times
  X)$. The derived wrapped Fukaya category $D^b\wW(T^*G)$ is generated
  by this cotangent fibre so, by passing to triangulated closures, we
  obtain a functor
  \[\mathfrak{C}\colon D^b\wW(T^*G)\to D^b\fF(X^-\times X).\]
  We call this the {\em moment correspondence functor}.
\end{dfn}

\subsection{Main results}

We now describe the situation we are interested in analysing.

\begin{setup}\label{stp}
  Let $G$ be a compact Lie group, let $(X,\omega)$ be a monotone
  symplectic manifold, and suppose that $X$ admits a Hamiltonian
  action of $G$ with equivariant moment map $\mu\colon
  X\to\mathfrak{g}^*$. Let $\kk$ be an arbitrary field of
  characteristic $p\geq 0$, where $G$ has no $p$-torsion if $p>0$. We
  will assume that $\mu$ vanishes transversely along $\mu^{-1}(0)$ and
  that $G$ acts freely on $\mu^{-1}(0)$. We will also assume that the
  Lagrangian submanifold
  \begin{equation}\label{eq:graphofzerolevel}M:=\{(x,y)\in X^{-}\times X\ :\ x,y \in \mu^{-1}(0),\ y=gx\mbox{ for some }g\in G\}\end{equation}
  is monotone.
\end{setup}

It is easy to check the following:

\begin{lma}\label{lma:mcf}
  Let $G$ denote the zero-section in $T^*G$. In the situation of Setup
  \ref{stp}, the generalised correspondence $(G,C)$ is composable and
  the geometric composition is $M$.
\end{lma}

\begin{lma}\label{lma:monotone}
  If $u\colon S^1\times[0,1]\to \left(T^*G\right)^-\times X^-\times X$
  is an annulus with $u(s,0)\in G\times (G\circ C)$ and $u(s,1)\in C$
  then
  \[\OP{area}(u)=\tau\OP{index}(u),\]
  where $\tau$ is the monotonicity constant for $X$.
\end{lma}
\begin{proof}
  We write $u(s,t)=(w(s,t),z(s,t))$ where $w(s,t)\in
  \left(T^*G\right)^-$ and $z(s,t)\in X^-\times X$. Let
  $K:=C\cap(G\times (G\circ
  C))=\{(g,0,x,gx)\ :\ x\in\mu^{-1}(0)\}$. Note that the map
  $\pi_1(K)\to\pi_1(C)$ is surjective and that $K\cong G\circ C$ is
  connected, so $\pi_1(C,K)=\{1\}$. This means that the loop $u(s,1)$
  in $C$ is homotopic in $C$ to a loop contained in $K$. Since we are
  only interested in $u$ up to homotopy, we can therefore assume that
  $u(s,1)$ is contained in $K$. Thus $w(s,t)$ is an annulus in $T^*G$
  with boundary on $G$. Since $T^*G$ retracts onto $G$, this is
  homotopic, through annuli in $T^*G$ with boundary on $G$, to the
  annulus $w'(s,t):=w(s,1)$. Therefore $u(s,t)$ is homotopic to an
  annulus $u'(s,t)=(w(s,1),z(s,t))$ with boundary on $G\times (G\circ
  C)$ and $K\subset C$. Note that $z$ is an annulus with boundary on $G\circ C$. The
  area and index of $u'$ agree with the area and index of $z$
  {\cite[Lemma 2.1.2]{WWcomposition}}. The boundary of $z$
  can be capped off in $X^-\times X$ since $\pi_1(X)=\{1\}$. Therefore
  monotonicity for $z$ follows from monotonicity for discs with
  boundary on $G\circ C\subset X^-\times X$.
\end{proof}

This monotonicity result allows us to incorporate $(G,C)$ as an object
of type (B) into the category $\fF_C(X^-\times X;\kk)$. As a
consequence of Lemmas \ref{lma:mcf} and \ref{lma:monotone}, we can
apply the geometric composition theorem in quilted Floer theory to
deduce that
\[\mathfrak{C}(G)\simeq M.\]

\begin{thm}\label{thm:generation}
  In the situation of Setup \ref{stp}, the Lagrangian $M$ in
  $\fF(X^{-}\times X)$ split-generates $\Delta_\alpha$ over $\kk$ if
  and only if $\HF(M,\Delta_\alpha;\kk)\neq 0$.
\end{thm}
\begin{proof}
  By Corollary \ref{cor:grps}, we know that, in $\wW(T^*G)$, the
  zero-section $G\subset T^*G$ can be written as a Koszul twisted
  complex built out of $T^*_1G$. By Lemma \ref{lma:functorkoszul}, the
  image of the Koszul twisted complex under the moment correspondence
  functor $\mathfrak{C}$ is a Koszul twisted complex representing
  $\mathfrak{C}(G)$ built out of $\mathfrak{C}(T^*_1G)$. By Lemma
  \ref{lma:mcf} and Theorem \ref{thm:formal}(i) we see that
  $\mathfrak{C}(T^*_1G)\simeq\Delta$ and $\mathfrak{C}(G)\simeq
  M$. Therefore $M$ is quasi-isomorphic to a Koszul twisted complex
  built out of $\Delta$. Lemma \ref{lma:diaggen} immediately implies
  the desired result.
\end{proof}

We first work out the implications of this theorem in the situation
where the $G$-action has a free Lagrangian orbit.

\begin{cor}\label{cor:split-gen}
  In the situation of Setup \ref{stp}, if $L=\mu^{-1}(0)$ is a free
  Lagrangian $G$-orbit with $\HF(L_\alpha,L_\alpha;\kk)\neq 0$ then
  $L_\alpha$ split-generates $D^{\pi}\fF(X;\kk)_{\alpha}$.
\end{cor}
\begin{proof}
  In the situation that $\mu^{-1}(0)$ is a Lagrangian orbit, the image
  of the zero-section under the moment correspondence functor is
  $L\times L$. We have
  \[\HF(L\times L,\Delta;\kk)=\HF(L,L;\kk)\]
  and each side decomposes under $1=\sum_{\alpha\in W}e_\alpha$ as
  \[\HF(L\times L,\Delta_\alpha;\kk)=\HF(L_\alpha,L_\alpha;\kk).\]
  Therefore the statement follows from Theorem \ref{thm:generation}.
\end{proof}

The next corollary shows that, if a summand of the Fukaya category
contains a nonzero object $L_\alpha$ arising as a free Lagrangian
$G$-orbit, then any other object in that summand split-generates the
summand. It also gives an upper bound on the rank of Floer cohomology of
$L_\alpha$ with any other object.

\begin{cor}\label{cor:othergeneration}
  In the situation of Setup \ref{stp}, if $L=\mu^{-1}(0)$ is a free
  Lagrangian $G$-orbit with $\HF(L_\alpha,L_\alpha;\kk)\neq 0$ and
  $L'\in D^{\pi} \fF(X;\kk)_\alpha$ is another nonzero object then
  $L'$ split-generates $D^{\pi}\fF(X;\kk)_\alpha$ and the inequality
  \[\hf(L_\alpha,L')\leq 2^{n/2}\sqrt{\hf(L',L')}.\]
  holds. Here, $\hf$ denotes the rank of $\HF$.
\end{cor}
\begin{proof}
  As in the proof of Theorem \ref{thm:generation}, $L_\alpha\times
  L_\alpha$ is quasi-isomorphic to a Koszul twisted complex built out
  of $\Delta_\alpha$. Under the functor $\Phi$, this says that
  $\Phi(L_\alpha\times L_\alpha)$ is quasi-isomorphic to a Koszul
  twisted complex $\mathcal{T}$ of functors built out of the identity
  functor. Since $\Phi(L_\alpha\times L_\alpha)(L')\simeq
  \CF(L_\alpha,L')\otimes L_\alpha$ and $\Phi(\Delta)(L')\simeq L'$,
  Lemma \ref{lma:twcpxfunctors} implies that $\CF(L_\alpha,L')\otimes
  L_\alpha$
  is quasi-isomorphic to a Koszul twisted complex $K$ built
  out of $L'$. This contains $L_\alpha$ as a summand (see Remark
  \ref{rmk:projfunc}) and is generated by $L'$, hence $L'$
  split-generates $L_\alpha$. Since $L_\alpha$ itself split-generates
  $D^{\pi}\fF(X;\kk)_\alpha$, we deduce that $L'$ split-generates
  $D^{\pi}\fF(X;\kk)_\alpha$.

  Computing the $\mu^1_\Tw$-cohomology of the space of morphisms from
  $K$ to $L'$ gives
  \[\HF(L',L_\alpha)\otimes \HF(L_\alpha,L')\cong H(\hom(K,L'),\mu^1_\Tw).\]
  Expanding all the cones in the Koszul twisted complex $K$, we see
  that it is built out of $2^n$ copies of $L'$. To see this, using the
  notation of Definition \ref{dfn:koszul}, we see that
  $K=\Cone{\TO{K_{n-1}[s_{n-1}]}{x_{n-1}}K_{n-1}}$ is quasi-isomorphic
  to the twisted complex
  \begin{equation}\label{eq:koszulcube1}
    K_{n-1}[s_{n-1}+1]\xdashrightarrow{-x_{n-1}[1]}K_{n-1}
  \end{equation}
  built out of two copies of $K_{n-1}$. Similarly $K_{n-1}$ is
  quasi-isomorphic to
  \begin{equation}\label{eq:koszulcube2}
    K_{n-2}[s_{n-2}+1]\xdashrightarrow{-x_{n-2}[1]}K_{n-2},
  \end{equation}
  so replacing each $K_{n-1}$ in \eqref{eq:koszulcube1} by
  \eqref{eq:koszulcube2}, we get a twisted complex for $K$ built out of
  four copies of $K_{n-2}$. Continuing in this manner, we find a
  twisted complex for $K$ built out of $2^n$ copies of $K_0=L'$.

  By definition of the category of twisted complexes, the chain group
  $\hom(K,L')$ is therefore a direct sum of $2^n$ shifted copies of
  $\CF(L',L')$. The differential $\mu^1_\Tw$ preserves the length
  filtration on this chain group, so we get a spectral sequence whose
  $E_1$-term is the direct sum of $2^n$ shifted copies of
  $\HF(L',L')$. The rank of $HF(L',L_\alpha)\otimes HF(L_\alpha,L')$,
  which equals $\hf(L',L_\alpha)^2$, is therefore bounded above by
  $2^n\hf(L',L')$, which implies the stated inequality.
\end{proof}

\begin{rmk}
  This inequality generalises to the case when $G$ is an arbitrary
  compact Lie group. In that situation, you can still represent the
  zero-section as a twisted complex built out of the cotangent fibre
  \cite{AbouzaidBasedLoops}. Abouzaid's construction produces a
  twisted complex out of any Morse function $f\colon G\to\RR$ with one
  copy of $T^*_1G$ for every critical point. The argument in the
  corollary shows that if $\HF(L_\alpha,L')\neq 0$ then $L'$
  split-generates $L_\alpha$ and that $\hf(L_\alpha,L')\leq
  \sqrt{m(G)\hf(L',L')}$ where $m(G)$ is the minimal number of
  critical points of a Morse function on $G$.
\end{rmk}

\begin{cor}\label{cor:aboucrit}
  In the situation of Setup \ref{stp}, if $L=\mu^{-1}(0)$ is a free
  Lagrangian $G$-orbit with $\HF(L_\alpha,L_\alpha;\kk)\neq 0$ then
  \[\HH^{\bullet}(D^{\pi}\fF(X;\kk)_\alpha)\cong \QH^{\bullet}(X;\kk)_\alpha.\]
\end{cor}
\begin{proof}
  The final part of the proof of {\cite[Corollary 3.11]{Smith}} shows
  that, if $\Delta_\alpha$ is split-generated by product Lagrangians,
  then $\HH^{\bullet}(D^{\pi}\fF(X;\kk)_\alpha)\cong
  \QH(X;\kk)_\alpha$. The proof of Theorem \ref{thm:generation}
  implies that $\Delta_\alpha$ is split-generated by
  $M_\alpha=L_\alpha\times L_\alpha$. This proves the corollary.
\end{proof}

\begin{rmk}
  The proofs also carry over if $M$ is equipped with a local system
  pushed forward along the moment correspondence from a local system
  on the zero-section. In the case when $M=L\times L$, and $L$ is a
  free $G$-orbit, this is easy to understand: if $E$ is a local system
  on $G$ and $\underline{\kk}$ is the trivial local system on $C$ then
  $E\circ\underline{\kk}$ is the local system on $L\times L$
  isomorphic to $E^*\boxtimes E$. In other words, if
  $\rho\colon\pi_1(L)\to\kk^\times$ is the monodromy of $E$ then the
  monodromy of $E\circ\underline{k}$ around a loop
  $(\gamma,\delta)\in\pi_1(L\times L)$ is
  $\rho(\gamma)^{-1}\rho(\delta)$.
\end{rmk}

The next corollary, which explains the full implications of Theorem
\ref{thm:generation}, uses the full strength of quilt theory, in
particular the existence of an $A_\infty$-functor
$\Phi\colon\fF(X^{-}\times Y)\to\nufun(\fF^{\#}(X),\fF^{\#}(Y))$ with
the property that the functors $\Phi(J\circ K)$ and
$\Phi(J)\circ\Phi(K)$ are quasi-isomorphic when the geometric
composition is transverse and embedded.

\begin{cor}\label{cor:nolagorb}
  In the situation of Setup \ref{stp}, suppose that (a)
  $\mu^{-1}(0)/G$ has the property that its derived extended Fukaya
  category is quasi-equivalent to its derived Fukaya category, and
  that (b) $\HF(M,\Delta_\alpha;\kk)\neq 0$. Then the $\alpha$-summand
  of the monotone Fukaya category $D^{\pi}\fF(X;\kk)_\alpha$ is
  split-generated by $G$-invariant Lagrangian submanifolds.
\end{cor}
\begin{proof}
  By Theorem \ref{thm:generation}, $M_\alpha$ split-generates
  $\Delta_\alpha$. Therefore the image of the corresponding functor
  $\Phi(M_\alpha)$ split-generates $D^{\pi}\fF(X;\kk)_\alpha$. We will
  show that $\Phi(M)$ factors through another functor whose image is
  split-generated by the $G$-inv\-ar\-i\-ant Lagrangian submanifolds.

  The subset $\mu^{-1}(0)$ is a fibred coisotropic (i.e. the
  symplectic reduction is a manifold and the isotropic leaves are the
  fibres of a fibre bundle over the symplectic reduction). Indeed, the
  isotropic leaves are precisely the $G$-orbits. Let $B$ denote the
  symplectic reduction $\mu^{-1}(0)/G$. Let
  $\varpi\colon\mu^{-1}(0)\to B$ denote the projection map. Let
  \[\Gamma:=\{(x,y)\in\ B^{-}\times X\ :\ \varpi(y)=x\in B\}\]
  denote the associated Lagrangian correspondence. This correspondence
  gives us a quilt functor
  $\Phi(\Gamma)\colon\fF(B;\kk)\to\fF^{\#}(X;\kk)$. It is easy to see
  that $M=\Gamma^T\circ\Delta^T_B\circ\Gamma\subset X^{-}\times
  X$. Therefore, by the geometric composition result for functors, the
  functor $\Phi(M)$ factors through $\Phi(\Gamma)$. The functor
  $\Phi(\Gamma)$ is defined on the full extended category
  $\fF^{\#}(B;\kk)$, but by assumption if we pass to the derived
  category we can consider just the objects in the image of
  $\Phi(\Gamma)$ applied to $\fF(B;\kk)$. The geometric composition of
  such an object with $\Gamma$ is a $G$-invariant Lagrangian
  submanifold: if $Q\subset B$ is a Lagrangian then
  $\Phi(\Gamma)(Q)=\varpi^{-1}(Q)$. This gives the corollary.
\end{proof}

\section{Examples and applications}\label{sct:examples}

\subsection{The quadric 3-fold}

Let $X$ be the quadric 3-fold. The fact that a real ellipsoid in $X$
split-generates the zero-summand of $\fF(X;\CC)$ was proved in
\cite{Smith}. We begin by recovering this result, as we found this
example very instructive in formulating the general theory.

The quadric 3-fold admits an $SU(2)$-action with a free Lagrangian
$SU(2)$-orbit $L$. The Lagrangian $L$ has minimal Maslov number 6,
since it is simply-connected and the ambient minimal Chern number is
3. The number 6 is twice the dimension of $L$, so there cannot be any
quantum contributions to the Floer differential. Therefore we have
$\HF^*(L,L\;\kk)\cong H^*(S^3;\kk)$ for any field. This certainly
implies that $M=L\times L$ is nonzero in $\fF(X^{-}\times X)$, so
split-generates some of the summands of $\Delta$.

Let $\kk$ be a field. The quantum cohomology $\QH(X;\kk)$
is \[\kk[H,E]/(H^2-2E,E^2-H)\] and the first Chern class is
$c_1(X)=3H$. The block decomposition of this ring depends on the
characteristic of $\kk$ and on whether the characteristic polynomial
of $H$ splits over $\kk$.

\begin{itemize}
\item If $\OP{char}\kk=2$ then the nilradical of $\QH(X;\kk)$ is the ideal
  $I$ generated by $H$ and $E$; therefore $\QH(X;\kk)$ is a local ring
  with maximal ideal $I$ and residue field $\kk$. In this case, the
  Fukaya category has only one summand and the quantum cohomology is
  not semisimple.
\item If $\OP{char}\kk\neq 2$ then the element $e_1=H^3/4$ is a
  nontrivial idempotent and the identity splits as a sum of orthogonal
  idempotents $1=e_0+e_1$, with $e_0=1-H^3/4$. (If $\OP{char}\kk=3$ or
  if $\lambda^3-2$ is irreducible over $\kk$ then this is a maximal
  splitting of the identity into idempotents; otherwise $H^3/4$ can be
  further decomposed as a sum of idempotents.) The summand
  $\QH(X;\kk)_1$ associated to $e_1$ is the ideal $(H,E)$ and the
  summand $\QH(X;\kk)_0$ associated to $e_0$ is the ideal $(1-H^3/4)$.
\end{itemize}

Note that, if $\OP{char}(\kk)\neq 2$, the element $E=H^2/2$ acts invertibly
on the summand $\QH(X;\kk)_{1}$ and annihilates the summand
$\QH(X;\kk)_{0}$.

\begin{prp}
 Let $M\subset X^{-} \times X$ be the Lagrangian submanifold defined
 in Equation \eqref{eq:graphofzerolevel}. Then
 $\HF(M,\Delta_\alpha;\kk)\neq 0$ if and only if
 $\alpha=0$. Equivalently, by Theorem \ref{thm:generation}, the
 Lagrangian orbit $L$ split-generates $D^{\pi}\fF(X)_0$ and represents
 the zero object in the other summands.
\end{prp}
\begin{proof}
  Since the minimal Maslov number of $L$ is 6, the grading on
  $\HF^*(L,L;\kk)$ can be taken in $\ZZ/6$. Moreover, the moment
  correspondence is itself simply-connected so has minimal Maslov
  number 6. This means that the moment correspondence functor
  $\mathfrak{C}$ is naturally a functor between $\ZZ/6$-graded
  categories. The Koszul twisted complex representing the zero-section
  of $SU(2)$ has the form
  \[\TODASH{T^*_1SU(2)[3]}{-x[1]}{T^*_1SU(2)}\]
  where $x\in C_{2}(\Omega
  SU(2);\kk)=CW^{-2}(T^*_1SU(2),T^*_1SU(2))$. The Koszul twisted
  complex representing $M=L\times L$ is
  \[\TODASH{\Delta[3]}{-\mathfrak{C}^1(x)[1]}{\Delta}\]
  As the moment correspondence functor preserves the $\ZZ/6$-grading,
  $\mathfrak{C}^1(x)\in \QH^{-2}(X;\kk)=\QH^4(X;\kk)$, so
  $\mathfrak{C}^1(x)$ is a multiple of $E\in \QH^4(X;\kk)$. Suppose
  that $\mathfrak{C}^1(x)=NE$ for some $N\in\kk$; since the whole
  theory makes sense over $\ZZ$, we know that $N$ is the reduction to
  $\kk$ of some fixed integer $N$. We claim that $N=\pm 1$.

  If we accept this for a moment, then we see the Koszul twisted
  complex splits as a direct sum of twisted complexes
  \[\TODASH{\Delta_{\alpha}[3]}{-Ee_\alpha[1]}{\Delta_{\alpha}}\]
  where $e_\alpha$ is the idempotent corresponding to the
  $\alpha$-summand. As we observed above, $Ee_0=0$ and $Ee_1$ is
  invertible in $\QH(X;\kk)_1$. Therefore we see that
  $\HF(M,\Delta_\alpha)=0$ unless $\alpha=0$, and $M$ split-generates
  $\Delta_0$ by Corollary \ref{cor:nilpgen}, so
  $\HF(M,\Delta_0;\kk)\neq 0$.

  It remains to prove that $N=\pm 1$. Suppose this were not the case,
  and let $p$ be a prime dividing $N$ and $\kk'$ be a field of
  characteristic $p$. Over this field, the Koszul twisted complex is
  \[\TODASH{\Delta[3]}{0}{\Delta}\]
  and $L\times L\simeq\Delta[3]\oplus\Delta$. Taking Floer cohomology
  of this twisted complex with $\Delta$ gives
  \begin{equation}
  \label{eq:hfrankcontradiction}\HF(L\times L,\Delta;\kk')\cong \HF(\Delta,\Delta;\kk')[-3]\oplus \HF(\Delta,\Delta;\kk').
  \end{equation}
  We know that $\HF(L\times L,\Delta;\kk')\cong \HF(L,L;\kk')$ which
  has rank 2. We also know that $\HF(\Delta,\Delta;\kk')=\QH(X;\kk')$
  has rank 2. This gives a contradiction to Equation
  \eqref{eq:hfrankcontradiction}.
\end{proof}

\begin{rmk}\label{rmk:summands}
  Note that the first Chern class of $X$ is $3H$. In the literature on
  Fukaya categories, it is common to decompose the Fukaya category
  into summands according to the eigenvalues of quantum product with
  $c_1$. This is a coarser decomposition than the one we use, which
  separates out summands according to the simultaneous eigenspaces of
  quantum multiplication by all semisimple elements of quantum
  cohomology. In the case of the quadric 3-fold and a field $\kk$ of
  characteristic 3, if one only uses eigenvalues of $c_1$, then there
  is only one piece in the decomposition, corresponding to the
  repeated eigenvalue zero. The Lagrangian $SU(2)$-orbit does not
  split-generate the whole of this Fukaya category, only the summand
  corresponding to the idempotent $e_0=1-H^3/4$.
\end{rmk}

\begin{rmk}
  Note that there is a monotone Lagrangian torus $T$ which is (a)
  disjoint from $L$ and (b) nonzero in the Fukaya category provided
  the characteristic of the ground field is not equal to 2. This torus
  is obtained by taking the standard nodal degeneration of $Q$ and
  observing that it is toric; symplectically parallel-transporting the
  barycentric fibre from the nodal quadric to a nearby smooth quadric
  gives the monotone Lagrangian torus $T$, whose superpotential is
  \[x+y+z+\frac{1}{xy}+\frac{1}{xz}.\]
  This follows from the general formula in {\cite[Theorem
      1]{NishinouNoharaUeda}} for the superpotential of a Lagrangian
  torus arising from a toric degeneration, together with the explicit
  calculation of this superpotential for polytope PALP 3 in the
  Fanosearch table of 3-dimensional reflective polytopes
  \cite{Fanosearch}. This torus is clearly disjoint from $L$, as $L$
  is the vanishing cycle of the nodal degeneration, and its behaviour
  in the Fukaya category is determined by the critical points of its
  superpotential. In characteristic not equal to 2, it inhabits the
  summands of the Fukaya category away from $\alpha=0$.
\end{rmk}

\begin{rmk}
  Smith proves split-generation of the 0-summand (in the weaker sense
  of $c_1$-eigenvalues) over $\CC$. For comparison, he uses Seidel's
  exact triangle to show that $L\times L$ generates the following
  twisted complex
  \[\TODASH{\Delta}{C}{\OP{Graph}(\tau_L)^5\simeq\Delta[-5].}\]
  The second term is the graph of fifth power of the Dehn twist in
  $L$, which is Hamiltonian isotopic to the identity in this case. The
  morphism $C$ is the ``cycle class'', defined by counting sections of
  a quadric Lefschetz fibration with five nodal fibres; he shows it is
  a multiple of the first Chern class. Of course, this cannot be true
  in characteristic 3, where $c_1=0$ and where we know that there are
  summands $\Delta_\alpha$ not split-generated by $L$ (so $C\neq 0$).
\end{rmk}

\subsection{Toric Fanos}

Let $X$ be a smooth toric Fano complex projective $n$-fold, and let
$L$ denote the monotone (barycentric) torus fibre. Let $\kk$ be an
algebraically-closed field and $W\colon H^1(L;\kk^{\times})\to\kk$ be
the superpotential of $L$. Each point in $H^1(L;\kk^{\times})$ defines
a $\kk^{\times}$-local system on $L$. If $\xi$ is a critical point of
$W$ then $\HF((L,\xi),(L,\xi);\kk)\cong H^*(T^n;\kk)$
{\cite[Proposition 13.2]{ChoOh}}, {\cite[Theorem 3.9]{FOOOtoric}},
{\cite[Proposition 3.3.1]{BiranCorneaEnumerative}}.

The critical set of $W$ is a 0-dimensional scheme whose ring of
functions (the {\em Jacobian ring}) is known to be isomorphic to
quantum cohomology. For a proof, see {\cite[Theorem 1.3]{FOOOtoric}}
for the statement over $\CC$ and \cite{JackSmith} for the statement in
arbitrary characteristic; we restrict to algebraically closed fields
here only to get an easy correspondence between $\kk$-points of the
critical scheme and elements of $H^1(L;\kk^{\times})$. Each
$\kk$-point $\xi$ of $\OP{Crit}(W)$ therefore gives us (a) a local
system $\xi$ on $L$ with non-vanishing Floer cohomology, and (b) a
local summand $QH(X;\kk)_\xi\subset QH(X;\kk)$, the local ring at
$\xi$. Indeed, $(L,\xi)$ is non-zero as an object in
$D^{\pi}\fF(X;\kk)_\xi$. Since $L$ is a free Lagrangian $T^n$-orbit,
$(L,\xi)$ split-generates $D^{\pi}\fF(X;\kk)_\xi$ by Theorem
\ref{thm:generation}.

\begin{cor}\label{cor:toricgen}
  For each nonzero summand of the Fukaya category of a smooth toric
  Fano variety, there exists a local system $\xi$ on the barycentric
  torus fibre $L$ such that $(L,\xi)$ split-generates this summand. In
  fact, by Corollary \ref{cor:othergeneration}, any Lagrangian which
  is nonzero in the summand $D^{\pi}\fF(X;\kk)_\xi$ split-generates this
  summand.
\end{cor}

\begin{rmk}
  Such generation results have been established when the
  superpotential has non-degenerate critical points in characteristic
  zero \cite{Ritter} and for Morse or $A_2$-singularities in
  characteristic $\neq 2,3$ \cite{Tonkonog}. Note that in Corollary
  \ref{cor:toricgen} there is no assumption on nondegeneracy of the
  critical points of the superpotential or on the characteristic of
  the ground field.
\end{rmk}

\begin{exm}\label{exm:psun}
  Projective spaces $\cp{n}$ are toric Fano, so Corollary
  \ref{cor:toricgen} applies to them. If $n=p^k-1$ and $\kk$ is an
  algebraically closed field of characteristic $p$ then
  $\QH(\cp{p^k-1};\kk)$ has no proper idempotent summands, so any
  Lagrangian with $\HF(L,L;\kk)\neq 0$ split-generates the whole
  Fukaya category. In $\cp{m^2-1}$ there is a Lagrangian copy of
  $PSU(m)$. This Lagrangian has minimal Maslov number $2m$ and, when
  $m=p^k$ for some prime $p\geq 3$, its cohomology over $\kk$ is
  {\cite[Th\'{e}or\`{e}me 11.4]{Borel}}
  \[\kk[y]/(y^{p^k})\otimes\Lambda(x_1,x_3,\ldots,x_{2p^k-3})\]
  where $|x_i|=i$, $|y|=2$. Since this is generated by cohomology
  classes of degree strictly less than $2p^k-2$ and the minimal Maslov
  number is $2p^k$, {\cite[Theorem 1.2.2(i)]{BiranCorneaRigidity}}
  implies that $\HF(L,L;\kk)\neq 0$, hence this Lagrangian
  split-generates the Fukaya category over $\kk$. See also Iriyeh's
  paper \cite{IriyehMaslov} where these examples are discussed.
\end{exm}

\begin{exm}\label{exm:realtoric}
  In a similar vein, we can prove that the real part $L$ of a toric
  Fano variety $X$ split-generates the Fukaya category when
  $\OP{char}(\kk)=2$, provided $L$ is orientable. This was proved in
  certain cases by Tonkonog \cite{Tonkonog}, including
  $\rp{2n+1}\subset\cp{2n+1}$ (his methods also enable him to study
  some non-orientable cases). The real part $L$ is known to have
  nonzero Floer cohomology \cite{Haug}. To show that $L$ generates
  every summand, we need an observation of Tonkonog {\cite[Theorem
      1.12]{Tonkonog}}: if $CO^0\colon\QH(X;\kk)\to HF(L,L;\kk)$ is
  the closed-open map then
  \[\ker(CO^0)=\ker(\OP{Frob})\]
  where $\OP{Frob}\colon\QH(X;\kk)\to\QH(X;\kk)$ is the Frobenius map
  $x\mapsto x^2$. Since $e_\alpha^2=e_\alpha$ for all idempotents
  $e_\alpha$ in $\QH(X;\kk)$, none of them live in the kernel of
  $CO^0$. Therefore $L_\alpha\neq 0$ for every summand $\alpha$, and
  $L$ split-generates the whole Fukaya category by Corollary
  \ref{cor:toricgen}.

  The reason we need to assume orientability despite working in
  characteristic 2 is that for a non-orientable Lagrangian $L$, the
  value of $\mathfrak{m}_0(L_\alpha)$ is not determined by $\alpha$:
  we only know that $2\mathfrak{m}_0(L_\alpha)$ is equal to
  $2\alpha(c_1(X)_s)$, which is an empty statement in characteristic
  2. For example $\rp{2}$ has $\mathfrak{m}_0(\rp{2})=0$ while
  $\mathfrak{m}_0(T^2)=1\mod 2$ for the Clifford torus $T^2$. We thank
  Dmitry Tonkonog for drawing attention to this. By contrast, for an
  orientable Lagrangian, the Maslov class $\mu\in H^2(X,L;\ZZ)$ is
  divisible by 2, so one can always find a pseudocycle in $X\setminus
  L$ representing this relative class and the proof of {\cite[Lemma
      2.7]{Sheridan}} shows that $\mathfrak{m}_0(L)$ is an eigenvalue
  of $c_1(X)$ without factors of 2.
\end{exm}

\subsection{Nonformality of quantum cohomology}

The diagonal Lagrangian $\Delta\subset X^{-}\times X$ is relatively
spin {\cite[Proof of Theorem 1.9(1)]{FOOOorient}}, and monotone if $X$
is monotone, so for any field $\kk$ there is a Fukaya-Floer
$A_\infty$-algebra $\CF(\Delta,\Delta;\kk)$. Since
$\HF(\Delta,\Delta;\kk)\cong \QH(X;\kk)$, we can transfer this to an
$A_\infty$-structure on $\QH(X;\kk)$ by homological perturbation. We
will denote this $A_\infty$-algebra by $\mathcal{QH}(X;\kk)$.

Our arguments for split-generation used formality of $C_*(\Omega
G;\kk)$ in a very serious way, to express the zero-section as a Koszul
twisted complex built out of the cotangent fibre. However, we did not
require $\mathcal{QH}(X;\kk)$ to be formal, and indeed when the
underlying algebra $\QH(X;\kk)$ is not semisimple, formality often
fails. Recall the following result:

\begin{thm}[\cite{AvramovIyengar}]\label{thm:avramoviyengar}
  Let $\kk$ be a field and $A$ a commutative $\kk$-algebra,
  finite-dimensional as a $\kk$-vector space. If $\HH^n(A)=0$ for two
  nonnegative values of $n$ of different parity (for example if
  $\HH^\bullet(A)$ is finite-dimensional over $\kk$) then $A$ is a
  product of separable field extensions of $\kk$. In particular, it is
  semisimple.
\end{thm}

\begin{rmk}
  Note that this theorem is about ungraded algebras. It is possible
  for a non-semisimple graded algebra to become semisimple when the
  grading is collapsed (if the idempotents have mixed degree). The
  quantum cohomology is naturally $\ZZ/2$-graded; it sometimes admits
  a more refined $\ZZ/2N$-grading, but we are only interested in the
  $\ZZ/2$-grading. In fact, we are only interested in $\ZZ/2$-graded
  algebras supported in even degree, which are semisimple as
  $\ZZ/2$-graded algebras if and only if they are semisimple as
  ungraded algebras.
\end{rmk}

\begin{cor}\label{cor:avramoviyengar}
  Let $\aA$ be a $\ZZ/2$-graded $A_\infty$-algebra over $\kk$
  supported in even degree whose cohomology $A:=H(\aA)$ is commutative
  and finite-dimensional as a $\kk$-vector space. If $A$ is not
  semisimple as a $\ZZ/2$-graded algebra then $\HH^{\bullet}(\aA)$ has
  infinite rank.
\end{cor}
\begin{proof}
  Recall that the Hochschild cohomology $\HH^r(A)^s$ of a
  $\ZZ/2$-graded algebra $A$ admits a $(\ZZ,\ZZ/2)$-bigrading $(r,s)$;
  the $r$-grading is the length of a Hochschild cochain. Given the
  hypotheses of the corollary, $A=A^0$ is finite-dimensional,
  commutative and not semisimple. By Theorem \ref{thm:avramoviyengar},
  the Hochschild cohomology group $\HH^r(A)^0$ is nonzero for
  infinitely many $r$.

  By contrast, while the Hochschild differential for the
  $A_\infty$-algebra $\aA$ still respects the length filtration, it
  might not have degree one. Only the combined grading $r+s\in\ZZ/2$
  survives on $\HH^{\bullet}(\aA)$. If the $A_\infty$-algebra $\aA$ is
  $A_\infty$-quasi-isomorphic to its cohomology algebra $A$ then the
  Hochschild cohomology $\HH^{\bullet}(\aA)$ is simply
  $\bigoplus_{\bullet=r+s\in\ZZ/2}\HH^r(A)^s$. Since this sum is taken
  over infinitely many values of $r$, we see that $\HH^{\bullet}(\aA)$
  has infinite rank.
\end{proof}

\begin{prp}\label{prp:toricfanononformal}
  If $X$ is a toric Fano manifold and $\kk$ is an algebraically closed
  field then $\HH^{\bullet}(\mathcal{QH}(X;\kk))$ is
  finite-dimensional over $\kk$.
\end{prp}
\begin{proof}
  Suppose $X$ is a toric Fano. Then $X^{-}\times X$ is also a toric
  Fano and, by Corollary \ref{cor:toricgen}, its Fukaya category is
  split-generated by the product of monotone toric fibres $T\times T$,
  equipped with various different local systems. The diagonal $\Delta$
  is a nonzero object in $\fF(X^{-}\times X;\kk)$ for any field and
  hence it split-generates a subset of the summands of the Fukaya
  category, say $\langle\Delta\rangle=\bigoplus_{\alpha\in
    A}D^{\pi}\fF(X^{-}\times X;\kk)_\alpha$ (where
  $\langle\Delta\rangle$ denotes the subcategory split-generated by
  $\Delta$). By Corollary \ref{cor:aboucrit}, the Hochschild
  cohomology of $\langle\Delta\rangle$ is equal to a subspace of the
  quantum cohomology of $X^{-}\times X$. Since $\QH(X^{-}\times
  X;\kk)$ is finite-dimensional over $\kk$, this implies that
  $\HH^\bullet(\langle\Delta\rangle)$ is finite-dimensional over
  $\kk$. Since the endomorphism $A_\infty$-algebra of $\Delta$ in
  $\langle\Delta\rangle$ is $\mathcal{QH}(X;\kk)$, we have
  $\HH^\bullet(\langle\Delta\rangle)=\HH^\bullet(\mathcal{QH}(X;\kk))$,
  which must therefore be finite-dimensional over $\kk$ as claimed.
\end{proof}

If $X$ is a toric Fano manifold then its quantum cohomology is
supported in even degrees. Therefore Corollary
\ref{cor:avramoviyengar} and Proposition \ref{prp:toricfanononformal}
together imply the following result.

\begin{cor}\label{cor:nonform}
  If $X$ is a toric Fano and $\kk$ is an algebraically closed field
  such that $\QH(X;\kk)$ is not semisimple as a $\ZZ/2$-graded
  algebra, then $\mathcal{QH}(X;\kk)$ is not quasi-isomorphic to
  $\QH(X;\kk)$.
\end{cor}

This highlights a sense in which the quantum cohomology of a toric
Fano still behaves ``semisimply'' provided you remember its
$A_\infty$-structure, even when it is not semisimple as a ring. Remark
\ref{rmk:othergeneration} is another hint in this direction.

\begin{exm}
  The simplest example is
  $\mathcal{QH}(\cp{1};\FF_2)=\FF_2[H]/(H^2+1)$ with $|H|=2$. If one
  passes to the strictly unital $A_\infty$ model for this, it is not
  hard to show that the first non-vanishing higher product is
  $\mu^4(H,H,H,H)=H$.
\end{exm}

\subsection{Grassmannians}

The Grassmannian $\OP{Gr}(n,2n)$ admits a $U(n)$-action with a free
Lagrangian orbit {\cite[Proposition 2.7]{NoharaUeda}} with minimal
Maslov number $2n$. We will show that, when $n$ is a power of 2, this
Lagrangian split-generates the Fukaya category over a field of
characteristic 2.

\begin{prp}\label{lma:grnsummand}
  If $n=2^s$ and $\kk$ is a field of characteristic 2 then
  $\QH(\OP{Gr}(n,2n);\kk)$ is a local ring; in other words, it only
  has one summand in its block decomposition.
\end{prp}
\begin{proof}
  If $\OP{char}(\kk)$ does not divide $r+1$, then the quantum
  cohomology ring of $\OP{Gr}(r,N)$ has a description as the Jacobian
  ring of a potential function $W\colon\kk^r\to\kk$. This description
  is due to Gepner \cite{Gepner} and the potential function is defined
  as follows. Let $\sigma_i(\underline{\lambda})$ denote the $i$th
  elementary symmetric polynomial in $r$ variables
  $\underline{\lambda}=(\lambda_1,\ldots,\lambda_r)$. Consider the map
  \[\Phi\colon \kk^r\to\kk^r,\quad \Phi(\underline{\lambda})=(\sigma_1(\underline{\lambda}),\ldots,\sigma_r(\underline{\lambda})).\]
  Define the function $\tilde{W}\colon\kk^r\to\kk$ by
  \[\sum_{i=1}^r\left(\frac{\lambda_i^{N+1}}{N+1}+(-1)^r\lambda_i\right).\]
  Since this is a symmetric polynomial, it factors as
  $\tilde{W}=W\circ\Phi$ for some function $W\colon\kk^r\to\kk$. This
  $W$ is Gepner's potential function (see also {\cite[Section
      11.3]{McDuffSalamon}}). The spectrum of the Jacobian ring is a
  0-dimensional scheme whose underlying variety is the set of critical
  points of $W$. The preimages under $\Phi$ of critical points of $W$
  are critical points of $\tilde{W}$. The critical points of
  $\tilde{W}$ are
  \[\{\underline{\lambda}\ :\ \lambda_i^N=1,\ i=1,\ldots,r\}.\]
  Suppose that $r=2^s$, $N=2^{s+1}$ and $\OP{char}(\kk)=2$. There is
  only one $N$th root of unity in $\kk$, namely 1. Therefore
  $\tilde{W}$ has a unique critical point
  $\underline{\lambda}=(1,\ldots,1)$. The image of this under $\Phi$
  is $(0,\ldots,0,1)$, because
  \[\sigma_i(1,\ldots,1)=\binom{2^s}{i}\equiv \begin{cases}
  0\mod 2&\mbox{ if }i\neq 0,2^s\\
  1\mod 2&\mbox{ if }i=0,2^s.
  \end{cases}\]
  Therefore the spectrum of $\QH(\OP{Gr}(2^s,2^{s+1});\kk)$ is a
  0-dimensional scheme whose underlying variety is a single point,
  $(0,\ldots,0,1)$. This means that the quantum cohomology
  $\QH(\OP{Gr}(2^s,2^{s+1});\kk)$ is a local ring.
\end{proof}

\begin{cor}\label{cor:grn}
  The homogeneous Lagrangian $U(2^s)\subset\OP{Gr}(2^s,2^{s+1})$
  split-gen\-er\-ates the Fukaya category in characteristic 2.
\end{cor}
\begin{proof}
  The Lagrangian $U(n)$ is the fixed point set of an antisymplectic
  involution on a Hermitian symmetric space {\cite[Remark
      7]{IriyehSakaiTasaki}}. Oh \cite{OhIII} tells us that this
  Lagrangian, equipped with the trivial local system, has nonzero
  self-Floer cohomology over a field of characteristic 2. By Corollary
  \ref{cor:split-gen}, it split-generates whichever summand of the
  Fukaya category it inhabits. Lemma \ref{lma:grnsummand} tells us
  that there is only one summand.
\end{proof}

\begin{rmk}
  The Lagrangian $U(n)\subset\OP{Gr}(n,2n)$ has minimal Maslov number
  $2n$ and the minimal holomorphic discs with boundary on $U(n)$ have
  been explicitly described {\cite[Proposition 4.1]{NoharaUeda}}. It
  should be possible, using this information, to work out for all $n$
  precisely which summands of the Fukaya category of $\OP{Gr}(n,2n)$
  are split-generated by $U(n)$, equipped with different local
  systems, in different characteristics.
\end{rmk}

\bibliographystyle{plain}
\bibliography{GGeneration.bib}

\end{document}